\documentclass[12pt]{article}
\usepackage{amsmath}
\usepackage{amsthm}
\usepackage{amssymb}
\usepackage{graphicx}
\usepackage{color}
\usepackage{epsfig}
\usepackage[latin1]{inputenc}
\def\subneq{\mathop{\raise 0.7ex \hbox{$\subset$}}\!\!\!\!\!\!{\raise -0.6ex\hbox{$\neq$}}\,}

\def\u{\underline}

\def\QQ{{\ \rlap {\raise 0.4ex \hbox{$\scriptscriptstyle |$}}\hskip -0.2em Q}}

\def\1{{1\hskip-0.25em{\rm l}}}

\def\CC{{\ \rlap{\raise 0.4ex \hbox{$\scriptscriptstyle |$}}\hskip -0.2em C}}

\def\sobre#1#2{\lower 1ex \hbox{ $#1 \atop #2 $ } }
\def\bajo#1#2{\raise 1ex \hbox{ $#1 \atop #2 $ } }

\def\ep{\varepsilon}

\def\p{\partial}

\def\O{\Omega}

\def\o{\omega}

\def\u2{{u^\ep \over \ep^2 }}
\def\u3{{\displaystyle {\bar u}^\ep \over \ep^2 }}

\newcommand{\sN}{{\scriptstyle N}}

\newcommand{\be}{ \begin{equation} }
\newcommand{\ee}{ \end{equation} }
\newcommand{\beay}{ \begin{eqnarray} }
\newcommand{\eeay}{ \end{eqnarray}}
\newcommand{\beayn}{ \begin{eqnarray*} }
\newcommand{\eeayn}{  \end{eqnarray*} }

\begin{document}

\newtheorem{theorem}{Theorem}
\newtheorem{lemma}{Lemma}
\newtheorem{proposition}{Proposition}
\newtheorem{corollary}{Corollary}
\newtheorem{remark}{Remark}
\newtheorem{assumptions}{Assumptions}
\title{A Rigorous Derivation of the Equations for the Clamped Biot-Kirchhoff-Love Poroelastic plate}
\author{{\bf Anna Marciniak-Czochra}\thanks{AM-C  was supported by ERC Starting Grant "Biostruct" 210680 and Emmy Noether Programme of German Research Council (DFG).}  \\Institute of Applied Mathematics, IWR and BIOQUANT\\
University of Heidelberg\\
Im Neuenheimer Feld 267,
69120 Heidelberg , Germany \and  {\bf Andro Mikeli\'c}  \thanks{E-mail:
{\tt Andro.Mikelic@univ-lyon1.fr}. The research of A.M. was partially supported by the  Programme Inter Carnot Fraunhofer from BMBF (Grant 01SF0804) and ANR.}
\\ Universit\'e de Lyon, CNRS UMR 5208,\\
  Universit\'e
Lyon 1, Institut Camille Jordan, \\   43, blvd. du 11 novembre 1918,
 69622 Villeurbanne Cedex, France}
\date{\today}

\maketitle

\begin{abstract} In this paper we investigate the limit behavior of the solution to quasi-static Biot's equations in thin poroelastic plates as the thickness tends to zero. We choose Terzaghi's time corresponding to the plate thickness and obtain the strong convergence of the three-dimensional solid displacement, fluid pressure and total poroelastic stress to the solution of the new class of plate equations. In the new equations the in-plane stretching { is} described by the 2D Navier's linear elasticity equations, with elastic moduli depending on  Gassmann's and Biot's coefficients.
The bending equation is coupled with the pressure equation and it contains the bending moment due to the variation in pore pressure across the plate thickness. The pressure equation {  is parabolic only in the vertical direction. As additional terms it } contains the time derivative of the in-plane Laplacean of the vertical deflection of the plate and {  of the the elastic in-plane compression term}.
\bigskip


{\bf AMS classcode}
35B25; 74F10; 74K20; 74Q15; 76S

\end{abstract}

\section{Introduction}

A plate is a 3D body bounded by two surfaces of small curvature and placed at small distance. A plate is said to be thin if the distance between these surfaces, called the thickness, is much smaller than a characteristic size of the surrounding surfaces.

The construction of a linear theory for the extensional and flexural deformation of plates, starting from the  3D Navier equations of linear elasticity, goes back to the 19th century and Kirchhoff's work. Following a short period of controversy, the complete theory, nowadays known as the Kirchhoff-Love equations for bending of thin elastic plates, was derived. Derivation was undertaken under assumptions, which are referred to as Kirchhoff's hypothesis. It reads as follows:
\vskip1pt
\fbox{ Kirchhoff's hypothesis:} Every straight line in the plate that was originally perpendicular to the plate midsurface, remains straight after the strain and perpendicular to the deflected midsurface.
\vskip1pt
Since then a theory providing appropriate 2D equations applicable to shell-like bodies was developed.
Later, due to the considerable difficulties with the derivation of plate equations from 3D linear elasticity equations, a direct approach in the sense of Truesdell's school of continuum mechanics, using a Cosserat surface, was proposed. We refer to the review paper of Naghdi  \cite{Naghdi} .
In classical engineering textbooks one finds a formal derivation, based on the Kirchhoff hypothesis. For details we refer to  Fung's textbook \cite{fung}.

A  different approach is to consider the plate equations as an approximation to the 3D elasticity equations in a plate domaine $\Omega^\ep = \o \times (-\ep , \ep )$, where $\o$ is the middle surface and $\ep$ is the ration between the plate thickness and its longitudinal dimension. Here, for simplicity we suppose that  plate is flat and of uniform thickness.

The comparison between 3D model and 2D equations was first performed formally in  \cite{FrDr:61}. Then Ciarlet and collaborators developed systematically the approach where the vertical variable $x_3 \in (-\ep , \ep )$ was scaled by setting $y_3= x_3 / \ep $. This change of variables transforms the PDE to a singular perturbation problem on a fixed domain. With such approach Ciarlet and Destuynder have established the rigorous error estimate between the 3D solution and the Kirchhoff-Love solution to the 2D plate equations, in the limit as $\ep \to 0.$ For details, we refer to the article \cite{CD:79}, to the book \cite{Ciarlet90} and to the subsequent work for details. The complete asymptotic expansion is due to Dauge and Gruais (see \cite{DG1}, \cite{DG2} and subsequent work by Dauge and collaborators). Handling general boundary conditions required the boundary layer analysis. For a review, discussing also the results by Russian school, we refer to \cite{DFY04}.
Further generalizations to nonlinear plates and shells exist and were obtained using $\Gamma -$ convergence. We do not discuss it here and don't undertake to give  complete references  to the plate theory.

Many living tissues are fluid-saturated thin bodies like bones,  bladders, arteries and diaphragms and they are  interpreted as poroelastic  plates or shells. For a review of modeling of bones as poroelastic plates we refer to \cite{Cow99}. Furthermore,  industrial filters are an example of poroelastic  plates and shells.

Our goal is to extend the above mentioned theory to the {\bf poroelastic} plates. These are the plates consisting of a porous material saturated by a viscous fluid.  The mathematical model of such a plate consists of the Biot poroelastic equations, instead of Navier's elasticity equations. In addition to the phase displacements, description of a poroelastic medium requires the pressure field and, consequently, an additional PDE.

Until recently, few articles have addressed plate theory for poroelastic media. Nevertheless, in the early paper \cite{Biot64}, Biot examined the buckling of a fluid-saturated porous slab under axial compression. This can be considered as the first study of a poroelastic plate. The goal was to have a model for the buckling of porous media, which is simpler than general poroelasticity equations. The model was obtained in the context of the thermodynamics of irreversible processes.

More recently, in \cite{TheBe94}
Theodorakopoulos and Beskos  used Biot's poroelastic theory and the Kirchhoff theory assuming thin plates
and neglected any in-plane motion to obtain purely bending vibrations. A systematic approach to the linear poroelastic plate and shell theory was undertaken by Taber in \cite{Tab92a} and \cite{TabPul96}, using 
 Kirchhoff's approach (see e. g. \cite{Szi74}).


In this paper we follow the approach of Ciarlet and Destuynder, as presented in the textbook \cite{SHSP92} and rigorously develop equations for a poroelastic plate.

\section{Setting of the problem}

We study the deformation and the flow in a poroelastic plate  $\Omega^\ell  = \{ (x_{1}, x_{2}, x_{3}) \in \omega_L \times (-\ell /2 , \ell /2 ) \}$, where  the mid-surface $ \omega_L$ is a bounded domain in $\mathbb{R}^2$ with a smooth
boundary $\partial \omega_L \in C^1$. For simplicity, we suppose that the poroelastic plate $\Omega^\ell$ is an isotropic material.
$\Sigma^{\ell}$ (respectively $\Sigma^{-\ell}$ ) is the upper face (respectively lower face) of the plate $\Omega^\ell$. $\Gamma^\ell$ is the lateral boundary, $\Gamma^\ell = \partial \omega_L \times (-\ell /2 , \ell /2 )$.
We recall that the ratio between the plate thickness and the characteristic horizontal length is $\ep= \ell / (2L) <<1$.

A poroelastic plate consists of an elastic skeleton (the solid phase) and pores saturated by a viscous fluid (the fluid phase). At the pore scale, one deals with a complicated fluid-structure problem and in applications we model it using the effective medium approach. In poroelasticity the effective modeling goes back to the fundamental work by Biot (see \cite{Biot55}, \cite{Biot63} and \cite{TOL}). The deformable porous media, saturated by a fluid, are modeled using Biot's diphasic equations for the effective solid displacement and the effective pressure. Biot's equations are valid at every point of the plate and the averaged phases coexist at every point. The two-scale poroelasticity equations were obtained using the two-scale expansions applied to the pore fluid-structure  equations by Burridge, Keller, Sanchez-Palencia, Auriault and many others.  We refer to the book \cite{SP-80}, the review \cite{Au-97} and the references therein. Mathematically rigorous justification of the two-scale equations is due to Nguetseng \cite{NG2} and to the papers by Mikeli\'c et al \cite{GM:00}, \cite{ClFGM} and \cite{FM:03}, where  it was also shown that the two-scale equations  are equivalent to the Biot system.
\vskip3pt

\begin{table}[ht]
{\footnotesize
\centerline{\begin{tabular}{|l|l|} \hline\hline
\emph{SYMBOL} & \emph{ QUANTITY }
  \\
\hline   $G$ & shear modulus     \\
\hline    $\nu$ & drained Poisson ratio   \\
\hline    $\gamma_G $  & inverse of Biot's modulus     \\
\hline    $\alpha$   & effective stress coefficient    \\
\hline    $k$  & permeability    \\
\hline    $\eta$  & viscosity    \\
\hline  $L$ and $\ell$  & midsurface  length and plate  width, respectively      \\
\hline    $\varepsilon = \ell  / (2L)$  & small parameter      \\
\hline     $T= \eta L^2 / (k G)$   &  characteristic Terzaghi's time    \\
\hline    $d$   & characteristic displacement      \\
\hline    $P=d G/L$  & characteristic  fluid pressure  \\
 \hline   $\mathbf{u} = (u_1 , u_2 , u_3 )$ & solid phase displacement \\
 \hline   $p$ &  pressure \\
 \hline
\end{tabular}}
} \caption{{\it Parameter and unknowns  description}\label{Data}}
\end{table}

We note that Biot's diphasic equations describe behavior of the system at so called Terzaghi's time scale $T=\eta L^2_c / (k G),$ where $L_c$ is the characteristic domain size, $\eta$ is dynamic viscosity, $k$ is permeability and $G$ is the bulk modulus. For the list of all parameters see the Table \ref{Data}.

In general, for the plate we have two possible choices of time scale:
\begin{enumerate}
  \item $L_c = L $, leading to $T= \eta L^2 / (k G)$ and
  \item $L_c =\ell $, leading to the  Taber-Terzaghi  transversal time $T_{tab} = \eta \ell^2 / (4k G)$.
\end{enumerate}
If Terzaghi's time is short, then it is necessary to study the vibrations of the poroelastic plate. For filters, Terzaghi's time scale doesn't correspond to short times, the rescaled acceleration factors $\max \{ \rho_f , \rho_s \} k^2 G/ (\eta^2 L_c) $ are small and the acceleration is negligible. Only the time change of the variation of fluid volume per unit reference volume $\zeta = \gamma_G p + \alpha \mbox{ div } {\bf u} $ is not small. For more discussion of the scaling we refer to \cite{MW:10}.

Consequently,   we  study the simplest model of real applied importance: the quasi-static Biot system.

Following Biot's classical work \cite{Biot55}, the governing equations as written in \cite{Cow99} take the following form:
\begin{gather}\sigma = 2G e({\bf u} ) + ( \frac{2\nu G}{1-2\nu} \mbox{ div }{\bf u}- \alpha p ) I  \; \mbox{ in } \; \Omega^\ell , \label{Coweq1} \\
-G \bigtriangleup {\bf u} - \frac{G}{1-2\nu} \bigtriangledown \mbox{ div }{\bf u} +\alpha \bigtriangledown p =0 \; \mbox{ in } \; \Omega^\ell ,  \label{Coweq1a} \\
\frac{\partial }{ \partial t} (\gamma_G p + \alpha \mbox{ div } {\bf u} ) - \frac{ k}{\eta} \triangle p =0 \; \mbox{ in } \; \Omega^\ell .
\label{Coweq2}
\end{gather}
We impose a given contact force $\sigma \mathbf{n} = \mathcal{P}^{\pm \ell} $  and a given normal flux $ \displaystyle -\frac{k}{\eta}\frac{\p p}{\p x_3} =U^\ell $ at $x_3 = \pm \ell /2$.

 We recall that the effective displacement of the solid phase is denoted by $\mathbf{u} = (u_1 , u_2 , u_3 ),$ the strain tensor $e$ is given by $e(\mathbf{u})=sym\bigtriangledown \mathbf{u}$, $\sigma$ is the stress tensor and the effective pore pressure is $p$.
\vskip5pt
Following  engineering textbooks approach to Kirchhoff-Love's plate model and with some appropriate modifications, we that are  able to derive {\bf formally} the poroelastic plate equations.

{  For this,
in addition to  the Kirchhoff hypothesis and following Taber's papers \cite{Tab92a} and \cite{TabPul96}, we suppose that
\begin{gather}
    \mbox{ the fluid velocity derivatives in the longitudinal direction} \notag \\
    \mbox{ are small compared to the transverse one.}\label{pressureapp}
\end{gather}}
Following  Fung's textbook \cite{fung}, we perform classical Kirchhoff type formal calculations   and obtain the following equations
\begin{gather}
     G\ell  \Delta_{x_1 , x_2 } \mathbf{u}^\omega + \frac{G\ell (1+\nu)}{1-\nu} \nabla_{x_1 , x_2 } \mbox{div}_{x_1 , x_2 } \mathbf{u}^\omega + \frac{\alpha (1-2\nu)}{1-\nu} \nabla_{x_1 , x_2 } N  \notag \\
     + \sum_{j=1}^2 (\mathcal{P}^\ell_j +\mathcal{P}^{-\ell}_j) \mathbf{e}^j =0      , \label{Strech1} \\
   {  (\gamma_G + \frac{\alpha^2  (1-2\nu)}{2G (1-\nu)} ) N  = \frac{ \alpha (1-2\nu ) \ell}{1-\nu} \mbox{ div}_{x_1 , x_2} (u^\o_1 , u^\o_2 ),} \label{Strech2A}\\
 { \hskip-12pt  (\gamma_G + \frac{\alpha^2  (1-2\nu)}{2G (1-\nu)} ) \frac{\partial }{ \partial t} \big( p^{eff} +\frac{N}{\ell}   \big)   -  \frac{k}{\eta}  \frac{\p^2 }{\p {  x_3^2 }} \big( p^{eff} +\frac{N}{\ell} \big) =\alpha {  x_3  \frac{1-2\nu}{1-\nu}} \frac{\p }{\p t} \Delta_{x_1 , x_2 } w ,}\label{Stretch3}\\
    {  \frac{G \ell^3}{6 (1-\nu)} \Delta_{x_1 , x_2}^2 w + \alpha \frac{1-2\nu}{1-\nu}  \Delta_{x_1 , x_2} \int^{\ell /2}_{-\ell/2} x_3 p^{eff} \ dx_3 =} \notag \\
    \frac{\ell}{2} \sum_{i=1}^2 \frac{\p }{\p x_i} (\mathcal{P}^\ell_i +  \mathcal{P}^{-\ell}_i ) +  \mathcal{P}^\ell_3 +  \mathcal{P}^{-\ell}_3 ,\label{bendingplate}
\end{gather}
{  where $w(x_1 , x_2, t)$ is the effective transverse displacement of the surface, $\mathbf{\tilde u}^\o = (u^\omega_1 , u^\omega_2 )$,  $ \displaystyle u^\omega_j (x_1 , x_2 ,t)  - x_3 \frac{\p w}{\p x_j} $ , $j=1,2$, are the effective in-plane solid displacements,  $p^{eff}$ is the effective fluid pressure and $N=\displaystyle -\int^{\ell /2}_{-\ell/2} p^{eff} \ {  d x_3 }$ is the effective stress resultant due to the variation in pore pressure across the plate thickness.}

We note that $ D=\displaystyle  \frac{G \ell^3}{6 (1-\nu)}$ is the flexural rigidity of the plate solid skeleton.
\vskip3pt
 We refer to Appendix for the detailed formal calculation.

The disadvantage of such approach to derive the equations is in using some {\it ad hoc} hypothesis.
The  assumptions (\ref{HypoKirchh}) that the stresses $\sigma_{i3}$, $i=1,2,3$
are negligible, can only be satisfied approximatively.  They are not even consistent with the fact that we use then averages which are non zero. To this classical difficulty of the plate theory, we add a new {\it ad hoc} assumption (\ref{pressureapp}) on the pressure field.\vskip2pt
As in the case of the justification of the elastic plate equations by Ciarlet et al, here also the two-scale asymptotic expansion approach gives the correct answer.
Note that the presence of the effective stress resultant $p^\o$, due to the variation in pore pressure across the plate thickness, causes stretching even if the boundary conditions for $(u^\o_1 , u^\o_2 )$ are homogeneous.


 Our goal is to extend the Kirchhoff-Love plate justification by Ciarlet et al and by Dauge et al to the poroelastic case.  Due to its structure, the quasi-static Biot system cannot be written  as a minimization problem. Therefore, we cannot apply  the approach  of  $\Gamma$-convergence to obtain the effective equations. Instead, we  start with a weak convergence result,  and then, after adding additional correction terms show the strong convergence. The obtained result is then of the same type as in
  $\Gamma$-convergence approach    i.e. they would coincide in the  case of linear elasticity. In poroelastic case, we have to deal with more involved norms related to the energy of the Biot system.


{  In subsection  \ref{subDE} we present the dimensionless form of the problem. Subsection \ref{secEU} contains a sketch of the proof of existence a unique smooth solution for the starting problem. In subsection \ref{subConv} we formulate our convergence results. Section \ref{scal} is consecrated to the introduction of the rescaled problem, posed on the domain $\Omega =\o \times (-1,1)$. Then in Section  \ref{convg6} we study convergence of the solutions to the rescaled problem, as $\ep \to 0$.  In short Sections \ref{strongcoup} and \ref{strongstressco} we prove the strong convergence for the corrected displacement and pressures and stresses, respectively. In Appendix \ref{append} we give an ad hoc derivation of the model, which follows mechanical engineering textbooks and which is justified a posteriori by our rigorous results.}

\section{Main Results}

\subsection{Dimensionless equations}\label{subDE}

We introduce the dimensionless unknowns and variable by setting
\begin{gather*}
    \gamma = \gamma_G G ; \quad  P = \frac{GU}{L} ; \quad \lambda = \frac{2 \nu}{1-2\nu} ;
     \quad {  T= \frac{\eta \ell^2 }{4k G} ;} \\
  U {\bf u}^\ep = {\bf u} ; \quad P p^\ep = p ; \quad {\tilde x} L =x; \quad {\tilde t} T= t .
\end{gather*}
After dropping wiggles the system (\ref{Coweq1})-(\ref{Coweq2}) becomes
\begin{eqnarray}
- \bigtriangleup {\bf u}^\varepsilon-\frac{1}{1-2\nu}\bigtriangledown \mbox{ div }{\bf u}^\varepsilon +\alpha \bigtriangledown p^\varepsilon =0 \quad \mbox{ in } \; \Omega^ \varepsilon \times (0,T), \label{Bioteq1}
\\
\sigma^\ep = 2 e({\bf u}^\ep ) + ( \frac{2\nu }{1-2\nu} \mbox{ div }{\bf u}^\ep- \alpha p^\ep ) I  \; \mbox{ in } \; \Omega^\ep \times (0,T), \label{Sig} \\
\frac{\partial }{ \partial t} (\gamma p^\varepsilon + \alpha \mbox{ div } {\bf u}^\varepsilon ) -{  \ep^2 } \triangle p^\varepsilon =0 \quad \mbox{ in } \; \Omega^ \varepsilon \times (0,T), \label{Bioteq2}
\end{eqnarray}
where  ${\bf u}^\varepsilon=({u}_1^\varepsilon,{u}_2^\varepsilon,{u}_3^\varepsilon)$ denotes the dimensionless displacement field and $p^\varepsilon $ the dimensionless pressure. We study a plate $\Omega^\ep$ with thickness $2\varepsilon =\ell /L$ and section $\omega =\o_L /L$.
It is described by
\begin{equation*}
\Omega^\varepsilon = \{ (x_{1}, x_{2}, x_{3}) \in \omega \times (-\varepsilon , \varepsilon) \},
\end{equation*}
  $\Sigma^{+\varepsilon}$ (respectively $\Sigma^{-\varepsilon}$ ) is the upper face (respectively the lower face) of the plate $\Omega^ \varepsilon$. $\Gamma^ \varepsilon$ is the lateral boundary, $\Gamma^ \varepsilon = \partial \omega \times (-\varepsilon , \varepsilon)$.

We suppose that a given dimensionless traction force is applied on $\Sigma^{+\varepsilon} \cup \Sigma^{-\varepsilon}$
and impose  the frictional boundary conditions on $\Gamma^ \varepsilon$:
\begin{gather}
\sigma^\varepsilon {\bf n}^\ep = (2 e({\bf u}^\varepsilon) -\alpha  p^\varepsilon I + \frac{2\nu }{1-2\nu}(\mbox{ div }{\bf u}^\varepsilon) I)(\pm {\bf e}^{3})=\notag \\
 \mathcal{P}_\ep = \ep^2  (\mathcal{P}_1 , \mathcal{P}_2 , 0) + \ep^3 \mathcal{P}_3 \mathbf{e}^3   \;  \mbox{ on }\; \Sigma^{+\varepsilon} \cup \Sigma^{-\varepsilon}, \label{FricBC}
\\
{\bf u}^\varepsilon =0,  \quad \mbox{ on }\: \Gamma^\varepsilon. \label{FricBC1}
\end{gather}

For the pressure $p^\ep$,
at the lateral boundary $\Gamma^\varepsilon = \partial \omega \ \times (-\varepsilon,\varepsilon)$ we
impose a given inflow/outflow flux $V$:
\begin{equation}\label{Bclateral}
-  \bigtriangledown p^\varepsilon \cdot {\bf n} = V(x_1,x_2,t) \: ; \quad \quad \mathbf {n} =
(n_1,n_2,0)
\end{equation}
and at $\Sigma^{+ \varepsilon } \cup \Sigma^{-\ep}$,  we set
\begin{gather}
- \frac{ \p p^\varepsilon}{\p x_3} = U^{1} (x_1,x_2,t) \quad \mbox{on } \; {\Sigma} ^{\pm \ep} . \label{Bctop}
\end{gather}
Finally, we need an initial condition for $p^\ep$ at $t=0$,
\begin{equation}\label{Bcbottom}
    p^\ep (x_1 , x_2 , x_3 , 0 ) = p_{in} (x_1 , x_2 ) \quad \mbox{ in } \; \Omega^ \varepsilon.
\end{equation}

\def\sN2{\scriptscriptstyle i,j+1}
\def\seJ{{\scriptscriptstyle {i+\frac{1}{2},j}}}

\noindent We write problem \eqref{Bioteq1}-\eqref{Bcbottom} in the variational form:\\ \\
Let $V^{\ep}=\left\{{\bf z}\in H^1(\Omega^{\ep})^3|  \quad \mathbf{z} |_{\Gamma^{\ep}}=0 \right\}.$
Find ${\bf u}^{\ep}  \in H^1(0,T,V^{\ep})$, $p^\ep \in H^1(0,T; H^1 (\Omega^{\ep}))$ such that it holds
\begin{eqnarray}
&&\int_{\Omega^\ep} 2  e({\bf u ^\ep }): e({\bf \varphi}) dx + \frac{2\nu }{1-2\nu} \int_{\Omega^\ep} \mbox{ div } {\bf u ^\ep } \mbox{ div }{\bf \varphi}dx -\alpha\int_{\Omega^\ep}
p ^\ep  \mbox{ div }{\bf \varphi}dx\nonumber\\&=&\int_{\Sigma^{+\ep}\cup\Sigma^{-\ep}} {\bf \mathcal{P}_{\ep}}{\bf \varphi}ds, \quad \mbox{ for every }{\bf \varphi}\in V^\ep, \mbox{ and } t\in (0,T),\label{variational1}
\end{eqnarray}
\begin{eqnarray}
&&\gamma \int_{\Omega^\ep} \partial_t p^\ep \zeta dx + \int_{\Omega^\ep} \alpha \mbox{ div } {\partial_t {\bf u}^\ep }\zeta dx +{  \ep^2 } \int_{\Omega^\ep}
\nabla p ^\ep  \nabla \zeta dx\nonumber\\&=&-{  \ep^2} \int_{\partial\Omega^{\ep}} {\bf \mathcal{N}_{\ep}}\zeta ds, \quad \mbox{ for every }{\zeta}\in H^1(\Omega^\ep),\label{variational2}
\end{eqnarray}
where
\begin{equation*}
\mathcal{N}^\ep=\left\{\begin{array}{ll}
V, &\mbox{ on } \Gamma^\ep,\\
U^{1}, &\mbox{ on } \Sigma^{\pm \ep}.
\end{array} \right.
\end{equation*}
\begin{eqnarray}
&&p^\ep |_{t=0} = p_{in}, \quad \mbox{ in } \Omega^\ep.\label{variational3}
\end{eqnarray}

\subsection{Existence and uniqueness for the $\ep$-problem}\label{secEU}

In this section we prove existence and uniqueness of a solution $\{{\bf u}^\ep,p^\ep\} \in$ \break $ H^{1}(0,T; V^\ep )\times H^{1}(0,T;H^1( \Omega^\ep ))$,  $p^\ep |_{t=0} = p_{in}$ of the problem \eqref{variational1} -\eqref{variational3}. We construct  a Galerkin approximation $\bf{u}^\ep \approx \bf{v}_N$, $p^\ep \approx p_N$ and show its convergence to $\{{\bf u}^\ep,p^\ep\} $ when $N\rightarrow +\infty$.

{

\begin{assumptions}\label{Hypoth}
$p_{in} \in H^2 (\Omega^\ep)$,
$\mathcal{P_{\ep}}\in H^{1}(0,T;H^1(\partial \Omega^\ep  \backslash \Gamma^\ep))$, $\mathcal{N^{\ep}}\in H^{1}_0 (0,T;L^2(\partial \Omega^\ep ))$.
\end{assumptions}

\begin{proposition}\label{epexist}  Let us suppose assumptions \ref{Hypoth} . Then problem (\ref{variational1})- (\ref{variational3}) has a unique solution $\{ \mathbf{u}^\ep , p^\ep \} \in H^1 (0,T;  V^\ep ) \times H^1 (0,T; H^1 (\Omega)).$
\end{proposition}

\begin{proof}
Let $\beta_j$ be an orthogonal basis of $V^\ep$ with respect to the scalar product $$({\bf f},{\bf g})_V=\int_{\Omega^\ep}2 e({\bf f}):e({\bf g})dx$$ and $\xi_j$ an orthogonal basis of $H^1(\Omega^\ep)$, orthonormal for $L^2 (\Omega^\ep)$. Then, the Galerkin approximation is
 \begin{eqnarray}
&&\int_{\Omega^\ep} 2  e({ \bf{v}_N }): e({\bf \varphi}) dx + \frac{2\nu }{1-2\nu} \int_{\Omega^\ep} \mbox{ div } { (\bf{v}_N) } \mbox{ div }{\bf \varphi}dx -\nonumber\\
 &&\hskip-15pt\alpha \int_{\Omega^\ep}
p_N  \mbox{ div }{\bf \varphi}dx =\int_{\Sigma^{+\ep}\cup\Sigma^{-\ep}} { \bf \mathcal{P}_{\ep}}{\bf \varphi}ds, \; \mbox{ for every }{\bf \varphi}\in V_N, \mbox{ and } t\in (0,T),\label{Gvariational1} \\
&&\gamma \int_{\Omega^\ep} \partial_t p_N \zeta dx + \int_{\Omega^\ep} \alpha \mbox{ div } {\partial_t (\bf{v}_N) }\zeta dx +\ep^2 \int_{\Omega^\ep}
\nabla p_N  \nabla \zeta dx\nonumber\\&&=-\ep^2 \int_{\partial\Omega^{\ep}} {\bf \mathcal{N}_{\ep}}\zeta ds, \quad \mbox{ for every }{\zeta}\in span\{\xi_1,...,\xi_N\},\label{Gvariational2} \\
&&p_N |_{t=0} = p_{in}^N,\label{Gvariational3}
\end{eqnarray}
where $ p_{in}^N$ is a projection of  $p_{in}$ to $span\{\xi_1,...,\xi_N\}$ and $V_N=span\{\beta_1,...,\beta_N\}.$

Therefore, ${\bf v_N}=\Sigma_{l=1}^N b_{lN}(t) {\bf \beta}_l (x)$ and $p_N=\Sigma_{l=1}^N s_{lN}(t) \xi_l(x)$. Using Korn's inequality we obtain that the bilinear form
$$\mathcal{A}({\bf f},{\bf g})_V=\int_{\Omega^\ep}2 e({\bf f}):e({\bf g})dx+\frac{2\nu }{1-2\nu} \int_{\Omega^\ep}\mbox{ div } {\bf f} \mbox{ div } {\bf g}dx$$ is $V^\ep$-elliptic. Let
\begin{gather*}
    A_{kj} = \int_{\Omega^\ep} (2 e(\beta_j ) : e(\beta_k) +  \frac{2\nu }{1-2\nu} \mbox{ div } \beta_j \mbox{ div } \beta_k ) \ dx , \; 1\leq k,j\leq 3, \\
    \mathcal{M}_{kj} = \int_{\Omega^\ep} \xi_j \mbox{ div } \beta_k  \ dx , \; \mathcal{K}_{kj} = \ep^2 \int_{\Omega^\ep} \nabla \xi_j \nabla \xi_k  \ dx , \; 1\leq k,j\leq 3.
\end{gather*}
Then, using the definition of ${\bf v}_N$ and $p_N$, we obtain
\begin{gather}
    A {\bf b}(t) -\alpha \mathcal{M} {\bf s} (t) =\mathcal{F}^N (t) , \label{Matr1}\\
    \gamma \frac{d}{d t} {\bf s} (t) +\alpha \mathcal{M}^\tau \frac{d}{d t} {\bf b} (t) +\mathcal{K} {\bf s} (t) = \mathcal{G}^N  , \quad   {\bf s} (0) = {\bf s}_{0N} , \label{Matr2}
\end{gather}
\begin{equation*}
\mbox{where} \quad {\bf b}(t)=\left[ \begin{array}{l}
b_{1N}(t)\\...\\b_{NN}(t)
\end{array}\right] \quad \mbox{and} \quad {\bf s}(t)=\left[ \begin{array}{l}
s_{1N}(t)\\...\\s_{NN}(t)
\end{array}\right].
\end{equation*}
After inserting \eqref{Matr1} into \eqref{Matr2}, we obtain a linear system of ordinary differential equations for ${\bf s}=(s_{1N}(t),...,s_{NN}(t))^T,$
 \begin{equation*}
     (\gamma I +\alpha^2 \mathcal{M}^\tau  A^{-1} \mathcal{M} )\frac{d}{d t} {\bf s} (t) +\mathcal{K} {\bf s} (t) = \mathcal{G}^N - \alpha \mathcal{M}^\tau  A^{-1}  \frac{d}{d t} \mathcal{F}^N (t) , \quad   {\bf s} (0) = {\bf s}_{0N} .
 \end{equation*}
 Since the $ \big( \mathcal{M}^\tau  A^{-1} \mathcal{M} \xi , \xi \big) \geq c ||  \mathcal{M} \xi ||_2 $, for all $\xi \in \mathbb{R}^N$,
 it has a unique solutions for all times. Moreover,  ${\bf s} \in H^{2}(0,T)$ implying the same time regularity for ${\bf v}_N$ and $p_N$. Next, we show {\it a priori} estimates, uniform in $N$. Testing \eqref{Gvariational1} by $\partial_t {\bf v}_N$, and \eqref{Gvariational2} by $p_N$, integrating in space and summing up, we obtain
\begin{eqnarray*}
&&\partial_t\left\{ \int_{\Omega^\ep} |e({\bf v}_N)|^2dx +  \frac{\nu }{1-2\nu}\int_{\Omega^\ep} |\mbox{ div } {\bf v}_N|^2dx+\frac{\gamma}{2}\int_{\Omega^\ep} |p_N|^2dx\right\} \\
 && + \int_{\Omega^\ep} \ep^2  |\nabla p_N(t)|^2 dx = \partial_t \{ \int_{\Sigma^{+\ep}\cup\Sigma^{-\ep}}  { \bf \mathcal{P}_{\ep}}{\bf v}_N ds -\ep^2 \int_{\partial\Omega^{\ep}} {\bf \mathcal{N}_{\ep}}p_N ds
\}\\
 &&- \int_{\Sigma^{+\ep}\cup\Sigma^{-\ep}}\partial_t {\bf \mathcal{P}_{\ep}}{\bf v}_Nds +\int_{\partial\Omega^{\ep}} \ep^2 \partial_t{\bf \mathcal{N}_{\ep}}p_N ds.
\end{eqnarray*}
It yields
\begin{eqnarray*}
&& \|{\bf v}_N \|_{L^{\infty}(0,T;V^\ep)}+ \|\nabla p_N\|_{L^{2}(0,T;L^{2}(\Omega^\ep))}+ \| p_N \|_{L^{\infty}(0,T;L^{2}(\Omega^\ep))}\\
 &\leq& C (\ep) \left\{ 
   \|{\bf \mathcal{P}}_{\ep}\|_{H^{1}(0,T;L^{2}(\partial\Omega^\ep \backslash\Gamma^\ep))}+  \|{\bf \mathcal{N}}^{\ep}\|_{H^{1}(0,T;L^{2}(\partial\Omega^\ep ))}\right\}.
\end{eqnarray*}
Next we use (\ref{Gvariational1})-(\ref{Gvariational2}) to calculate $ \p_t {\bf v}_N$ and $\p_t p_N$ at $t=0$.
Taking the time derivative provides $H^1$-regularity in time. Finally, passing to the limit $N\rightarrow \infty$ and uniqueness are obvious.

\end{proof}
}

\subsection{Convergence results}\label{subConv}
In the remainder of the paper we make the following assumptions
\begin{assumptions}\label{Hypoth1}
 {  For simplicity, we assume that $p_{in} =0$, that $V$ has a compact support in $\p \o \times (0,T]$ and
 that $U^{1}$  and $\mathcal{P_{\ep}}$ have compact support in $\omega \times (0,T]$.}
\end{assumptions}

We start by formulating the scaled limit displacements and pressures:

Let the scaled in-plane displacements and the pressure $\{ {\bf \tilde w}^0 , \pi_m \} \in H^1 (0,T; H^1_0 (\o ))^3$ be given by
\begin{gather}
     (\gamma  + \frac{\alpha^2  (1-2\nu)}{2 (1-\nu)} ) \pi_m  + \frac{ \alpha (1-2\nu )}{1-\nu} \mbox{ div}_{y_1 , y_2} (w^0_{1} , w^0_{2} ) =0, \; 
     \mbox{ in } \; \o \times (0,T); \label{StrechJ01Tcor} \\
    -\Delta_{y_1 , y_2} ( w^0_{1} , w^0_{2} ) - \frac{1+\nu}{1-\nu} \nabla_{y_1 , y_2} \mbox{ div}_{y_1 , y_2} (w^0_{1} , w^0_{2} )  +
     \frac{\alpha (1-2\nu)}{1-\nu} \nabla_{y_1 , y_2} \pi_m \notag \\
     = {  \frac{1}{2} }{  ( \mathcal{P}_1 |_{y_3 =1} + \mathcal{P}_1 |_{y_3 =-1} , \mathcal{P}_2 |_{y_3 =1} + \mathcal{P}_2 |_{y_3 =-1} )}, \quad
     \mbox{ in } \; \o \times (0,T) . \label{StrechJ04Tcor}
\end{gather}
Theory of the stationary 2D Navier system of linear elasticity yields
\begin{lemma} Under assumptions (\ref{Hypoth})-(\ref{Hypoth1}), problem (\ref{StrechJ01Tcor})-(\ref{StrechJ04Tcor}) has a unique solution $\{ {\bf \tilde w}^0 , \pi_m \} \in H^{1} (0,T;  H^1_0 (\o ))^3 $.
\end{lemma}
The pressure fluctuation $\pi_w \in H^1 (0,T; H^1 ( \O ))$  is given by
\begin{gather}
      (\gamma + \alpha^2 \frac{1-2\nu}{2(1-\nu)} )   \p_t \pi_w - \frac{\p^2  \pi_w}{\p y_3^2}   -
     \alpha \frac{1-2\nu}{1-\nu} y_3 \Delta _{y_1 , y_2 } \p_t  w^0_{3} =0  \mbox{ in }  \Omega \times (0,T),  \label{presstabD1cor} \\
    \p_{y_3}  \pi_w |_{y_3 =1} = \p_{y_3}  \pi_w  |_{y_3 =-1} =-U^1 \ \mbox{ in } \  (0,T),  
     \quad \pi_w   |_{t=0} =0 \ \mbox{ in } \ \Omega .  \label{presstabD3cor}
\end{gather}
and the vertical displacement $w^0_3 \in H^1 (0,T; H^2_0 (\o)) $ by
\begin{gather}
    \frac{4}{3} \frac{1}{1-\nu} \Delta^2_{y_1 , y_2 } w^0_{3} +\frac{\alpha (1-2\nu )}{1-\nu} \Delta_{y_1 , y_2 } \int^1_{-1} y_3 \pi^0 \ dy_3 =\mathcal{P}_{3} |_{y_3=1} +\mathcal{P}_{3} |_{y_3=1} +\notag \\
     \mbox{ div }_{y_1 , y_2} (\mathcal{P}_{1} |_{y_3=1} + \mathcal{P}_{1} |_{y_3=-1} , \mathcal{P}_{2}  |_{y_3=1} + \mathcal{P}_{2} |_{y_3=1})  \quad \mbox{ in } \quad  \Omega \times (0,T),
     \label{Bendtab0cor1}
     \end{gather}
     \begin{proposition} Under assumptions (\ref{Hypoth})-(\ref{Hypoth1}), problem (\ref{presstabD1cor})- 
     (\ref{Bendtab0cor1}) has a  unique solution $ \{ w^0_{3}, \pi_w  \} \in H^1 (0,T; H^2_0 (\o )) \times H^1 ( 0,T; L^2_0 (\Omega )), \; \p_{y_3} \pi_w \in H^{1} ( 0,T; L^2 (\Omega))$, $\int^1_{-1} \pi_w \ dy_3 =0.$
\end{proposition}
\begin{proof} See Proposition \ref{Wpbending}.
\end{proof}
\begin{theorem} \label{th1} ({\bf Convergence of the displacement and the pressure}) $\; $ Let $ \{ {\bf \tilde w}^0 , \pi_m \} $   be given by  (\ref{StrechJ01Tcor})-(\ref{StrechJ04Tcor}) and $\{ w_3^0 , \pi_w \} $ by (\ref{presstabD1cor})-(\ref{Bendtab0cor1}). Furthermore let
\begin{gather}
   E_{cor}^0 (x_1 , x_2 , x_3 ,t)= \frac{\alpha (1-2\nu )}{2(1-\nu )} (\pi_w (x_1 , x_2 , \frac{x_3}{\ep} , t) +\pi_m (x_1 , x_2 ,t))
   -\notag \\
   \frac{\nu}{1-\nu } (\mbox{ div }_{x_1,x_2}{\bf \tilde w^0} (x_1 , x_2 , t) - \frac{x_3}{\ep} \Delta_{x_1 , x_2} w_3^0 (x_1 , x_2 , t)).\label{defE0}
\end{gather}
Then we have the following convergences
\begin{gather}
   \lim_{\ep \to 0} \max_{0\leq t \leq T }\frac{1}{|\Omega^\ep |}\int_{\Omega^\ep}  \ep^{-2} | \frac{\p u_3^\ep}{\p x_3}  -   \ep  E^0_{cor}|^2 \ dx =0,\label{Conv1}
\\
   \lim_{\ep \to 0} \max_{0\leq t \leq T }  \frac{1}{|\Omega^\ep |}\int_{\Omega^\ep} | e_{ij} \bigg( \frac{{\bf u}^\ep }{\ep} - {\bf w}^0
  + \frac{ x_3}{\ep} \nabla_{x_1,x_2} w_3^0  \bigg)
  |^2 \ dx =0, \; 1\leq i,j \leq 2, \label{Conv2} \\
\lim_{\ep \to 0} \max_{0\leq t \leq T }   \frac{1}{|\Omega^\ep |}\int_{\Omega^\ep} | \frac{ e_{j3} \big( {\bf u}^\ep \big)}{\ep}  |^2\ dx =0, \; 1\leq j \leq 2, \label{Conv3} \\
    \lim_{\ep \to 0} \max_{0\leq t \leq T }  \frac{1}{|\Omega^\ep |}\int_{\Omega^\ep} | \frac{u^\ep_j }{\ep} - w^0_j 
    + \frac{ x_3}{\ep} \nabla_{x_1,x_2} w_3^0   |^2\ dx =0, \; 1\leq j \leq 2,\label{Conv3a} \\
    \lim_{\ep \to 0} \max_{0\leq t \leq T }  \frac{1}{|\Omega^\ep |}\int_{\Omega^\ep} | u^\ep_3 - w^0_3 (x_1 , x_2 , t) |^2 \ dx =0,\label{Conv4a} \\
    \lim_{\ep \to 0} \max_{0\leq t \leq T } \frac{1}{|\Omega^\ep |}\int_{\Omega^\ep} | \frac{p^\ep  }{\ep} - \pi_w (x_1 , x_2 , \frac{x_3}{\ep} ,t) - \pi_m (x_1 , x_2 ,t) |^2 \ dx =0,\label{Conv4}
\\
   \lim_{\ep \to 0} \max_{0\leq t \leq T } \frac{1}{|\Omega^\ep |}\int_{\Omega^\ep} \{ | \nabla_{x_1 , x_2}  p^\ep |^2 +  | \p_{x_3} p^\ep - \p_{y_3} \pi_w |_{y_3 = x_3 /\ep}|^2 \}\ dx =0.\label{Conv5}
\end{gather}
\end{theorem}
\begin{proof} It is a direct consequence of Proposition \ref{proofstrong} . It is enough to rescale back the variables and the unknowns.
\end{proof}
\begin{remark} Convergence (\ref{Conv5}) justifies Taber's hypothesis (\ref{pressureapp}). \end{remark}
\begin{theorem} \label{th2} ({\bf Convergence of Biot's stress}) $\; $ Let us suppose assumptions from Theorem \ref{th1}.
Then we have
\begin{gather}
\lim_{\ep \to 0} \max_{0\leq t \leq T } \frac{1}{|\Omega^\ep |}\int_{\Omega^\ep} |
\frac{\sigma_{jj}^\ep   }{\ep} -2 e_{jj} (\mathbf{w}^0  ) +
2 \frac{x_3}{\ep} \frac{\p^2 w^0_3 }{\p x_j^2}  + 2 E^0_{cor} |^2 \ dx =0, \; j=1,2, \label{D02O}  \\
 \lim_{\ep \to 0} \max_{0\leq t \leq T } \frac{1}{|\Omega^\ep |}\int_{\Omega^\ep} | \frac{\sigma_{12}^\ep  }{\ep} -2 e_{12} (\mathbf{w}^0  ) +
2 \frac{x_3}{\ep} \frac{\p^2 w^0_3 }{\p x_1 \p x_2} |^2 \ dx =0,
\label{D03O} \\
\lim_{\ep \to 0} \max_{0\leq t \leq T } \frac{1}{|\Omega^\ep |}\int_{\Omega^\ep} | \frac{\sigma_{j3}^\ep }{\ep}   |^2 \ dx =0, \; j=1,2,3.
\label{D04O} 
\end{gather}\end{theorem}
\begin{proof} It is a direct consequence of Proposition \ref{stresscon} . It is enough to rescale back the variables and the stresses.
\end{proof}
\begin{remark} We note that convergences (\ref{D04O}) justify Kirchoff's hypothesis for the case of the poroelastic plate.
\end{remark}


\section{The scaled problem}\label{scal}

 Our objective is to study the behavior of the displacement ${\bf u}^{\ep}$ and the pressure $p^\ep$ in the limit $\ep \rightarrow 0$. Since the problem is defined on the domain $\Omega^\ep=\omega \times (-\ep,\ep)$, which is changing with $\ep$, it is convenient to transform it to a problem defined on a fixed $\Omega=\omega \times (-1,1)$.

To the unknowns ${\bf u}^\ep$ and $p^\ep$, defined on  $\Omega^\ep$, we associate the scaled displacement and pressure fields ${\bf w}(\ep)$ and $\pi (\ep)$ by the scalings $y_j=x_j$ for $j=1,2$ and $y_3=\frac{x_3}{\ep}$ with $x\in \Omega^\ep$, $y\in \Omega$ such that
\begin{eqnarray}\label{scaling}
w_j(\ep)(y_1,y_2,y_3,t)&=&\frac{u_j^{\ep}(x_1,x_2,x_3,t)}{\ep},\quad \mbox{ for } j=1,2,\nonumber\\
w_3(\ep)(y_1,y_2,y_3,t)&=& u_3^{\ep}(x_1,x_2,x_3,t), \nonumber\\
\pi (\ep)(y_1,y_2,y_3,t)&=&\frac{p^{\ep}(x_1,x_2,x_3,t)}{\ep}.
\end{eqnarray}
By direct calculation we find out that the scaled displacement and pressure fields satisfy the following variational problem:

Let $V(\Omega)=\left\{{\bf z}\in H^1(\Omega)^3|  \quad {\bf z}|_{\Gamma}=0  \right\},$ where $\Gamma=\partial \omega \times (-1,1)$ and denote ${\bf \tilde w}= (w_1, w_2)$ and  $\tilde \varphi= (\varphi_1, \varphi_2)$.\vskip1pt

Find ${\bf w}(\ep) \in H^1(0,T,V(\Omega))$, $\pi (\ep) \in H^1(0,T; H^1(\Omega))$,  such that it holds
 \begin{eqnarray}
&&2 \int_{\Omega}\sum_{i,j=1}^{2}   e_{ij}({\bf w} (\ep)) e_{ij}({\bf \varphi}) dy + \frac{2\nu }{1-2\nu} \int_{\Omega} \mbox{ div}_{y_1,y_2}  {\bf \tilde w} (\ep) \mbox{ div}_{y_1,y_2}{\bf \tilde  \varphi}dy \nonumber\\&&-\alpha\int_{\Omega}
\pi(\ep)  \mbox{ div}_{y_1,y_2}{\bf \tilde  \varphi}dy
+\frac{1}{\ep^2}\Big\{4 \int_{\Omega}\sum_{j=1}^{2}   e_{3j}({\bf w} (\ep)) e_{3j}({\bf \varphi}) dy \nonumber\\&& - \alpha\int_{\Omega}
\pi(\ep)  \partial_{y_3}{ \varphi_3}dy+ \frac{2\nu }{1-2\nu} \int_{\Omega} \mbox{ div}_{y_1,y_2}  {\bf \tilde w} (\ep) \partial_{y_3}{ \varphi_3}dy\nonumber\\
&&+ \frac{2\nu }{1-2\nu} \int_{\Omega} \partial_{y_3}  {w_{3} (\ep) } \mbox{ div}_{y_1,y_2}{\bf \tilde  \varphi}dy\Big\}+ \frac{2(1-\nu ) }{(1-2\nu )\ep^4} \int_{\Omega} \partial_{y_3}  { w_{3} (\ep) } \partial_{y_3}{\bf  \varphi_3}dy\nonumber\\&
=&  \sum_{j=1}^{2}  \int_{\Sigma^{+}\cup\Sigma^{-}}{  \mathcal{P}_{j}}{ \varphi_j}ds+  \int_{\Sigma^{+}\cup\Sigma^{-}} {  \mathcal{P}_{3}}{ \varphi_3}ds \; \forall {\bf \varphi}\in V,  t\in (0,T),\label{ScaledvariationalT1cortab} \\
&&\gamma \int_{\Omega} \partial_t \pi(\ep) \zeta dy + \int_{\Omega} \alpha \mbox{ div}_{y_1,y_2}  {\partial_t \bf \tilde w} (\ep) \zeta dy +\ep^2 \int_{\Omega} \nabla_{y_1,y_2}  \pi(\ep)  \nabla_{y_1,y_2} \zeta dy\nonumber\\&&+\alpha  \ep^{-2}  \int_{\Omega} \partial_{y_3}  {\partial_t w_{3} (\ep) } \zeta \ dy + \int_{\Omega} \frac{\partial \pi(\ep)}{\partial y_3}\frac{\partial \zeta}{\partial y_3} \ dy=- \ep \int_{-1}^1 \int_{\partial\omega} V\zeta ds\nonumber\\
&&  -  \int_{\Sigma^{+}\cup\Sigma^{-}} \pm U^1 \zeta \  ds , \quad \forall {\zeta}\in H^1(\Omega),\label{ScaledvariationalT2cor} \\
&& \qquad \qquad \pi (\ep) |_{t=0} = 0, \quad \mbox{ in } \Omega.\label{ScaledvariationalT3}
\end{eqnarray}
\section{Convergence of the scaled displacement and the pressure as $\ep\rightarrow 0$}\label{convg6}


\begin{proposition}\label{prop1T}
The following  {\it a priori} estimates hold
\begin{eqnarray}
 \|{\bf w}(\ep)\|_{H^{1}(0,T;H^{1}(\Omega)^3)}&\leq& C,\label{apriori1T}\\
 \| \pi (\ep)\|_{H^{1}(0,T;L^{2}(\Omega))}&\leq& C,\label{apriori2T}\\
  \|e_{ij}({\bf w}(\ep))\|_{L^{2}(\Omega \times(0,T))}&\leq& C,\quad \quad1\leq i,j\leq2,\label{apriori3T}\\
  \|e_{i3}({\bf w}(\ep))\|_{L^{2}(\Omega \times(0,T))}&\leq& C\ep,\quad \quad 1\leq i \leq2,\label{apriori4T}\\
    \|e_{33}({\bf w}(\ep))\|_{L^{2}(\Omega \times(0,T))}&\leq& C\ep^2,\label{apriori5T}\\
      \|\partial_{y_3} \pi (\ep)\|_{L^{2}(\Omega \times(0,T))}&\leq& C, \label{apriori6T} \\
       \|\nabla_{y_1 , y_2} \pi (\ep)\|_{L^{2}(\Omega \times(0,T))}&\leq& \frac{C}{\ep}. \label{apriori7T}
      \end{eqnarray}
\end{proposition}
\begin{proof}
Testing equations \eqref{ScaledvariationalT1cortab} and \eqref{ScaledvariationalT2cor} by $\varphi= \partial_t {\bf w}(\ep)$ and $\zeta=\pi (\ep)$, respectively, integrating them and summing up, we obtain
\begin{eqnarray}\label{33}
&&\frac{1}{2} \frac{d}{dt} \Big\{ 2 \int_{\Omega}\sum_{i,j=1}^{2}   \left |e_{ij}({\bf w(\ep) })\right|^2 dy + \frac{2\nu }{1-2\nu} \int_{\Omega} \left | \mbox{ div}_{y_1,y_2}  {\bf \tilde w(\ep) } \right |^2 \ dy + \gamma \int_{\Omega} \left | \pi (\ep)\right|^2dy\nonumber\\
&&+\frac{4}{\ep^2}  \int_{\Omega}\sum_{j=1}^{2}  \left | e_{3j}({\bf w(\ep) })\right |^2 dy +  \frac{4\nu }{\ep^2 (1-2\nu )} \int_{\Omega} \mbox{ div}_{y_1,y_2}  {\bf \tilde w(\ep) } \partial_{y_3}{u_3(\ep)}dy\nonumber\\
&&+  \frac{2(1-\nu) }{(1-2\nu) \ep^4} \int_{\Omega}\left | \partial_{y_3}  w_3(\ep)\right|^2 dy \Big\}+\ep^2  \int_{\Omega} \left | \nabla_{y_1,y_2}  {\pi (\ep) } \right |^2dy + \int_{\Omega} \left | \frac{\partial \pi (\ep) }{\partial y_3}  \right |^2dy \nonumber\\
&& =\frac{d}{dt} \Big\{ \sum_{j=1}^{2}  \int_{\Sigma^{+}\cup\Sigma^{-}}{ \mathcal{P}_{j}}{w_j}(\ep)ds+  \int_{\Sigma^{+}\cup\Sigma^{-}} { \mathcal{P}_{3}}{ w_3(\ep)} \ ds \Big\} \nonumber
\\&& -\sum_{j=1}^{2}  \int_{\Sigma^{+}\cup\Sigma^{-}}{ \partial_t \mathcal{P}_{j}}{w_j}(\ep) ds-  \int_{\Sigma^{+}\cup\Sigma^{-}} { \partial_t \mathcal{P}_{3}}{ w_3(\ep)} \ ds    \nonumber \\
&& -  \int_{\Sigma^{+}\cup\Sigma^{-}}  \pm  U^1 \pi (\ep) \  ds - \ep \int_{-1}^1 \int_{\partial\omega}
V \pi (\ep) ds .
\end{eqnarray}
We can estimate the terms on the right hand-side of \eqref{33}:
\begin{eqnarray}
\left | \int_{\Sigma^{+}\cup\Sigma^{-}} {  \pm U^1} \pi (\ep)  ds\right| = \left |  \int_{\Omega} \frac{\partial}{\partial{y_3}} ( U^1 \pi (\ep))  ds\right| \leq C \int_{\Omega}\left | \frac{\partial}{\partial{y_3}}\pi (\ep)\right |  ds
\end{eqnarray}
\begin{gather}
\left | \int_{\Sigma^{+}\cup\Sigma^{-}} { \mathcal{P}_{3}}{ w_3(\ep)}ds\right| = \left | \int_{\Omega} \partial_{y_3} \left(w_3(\ep)\left( \frac{P_1-P_2}{2}y_3+\frac{P_1+P_2}{2}\right)\right)  dy\right| \leq \notag \\
C \int_{\Omega}(\left | \partial_{y_3}w_3(\ep)\right | + \left | w_3(\ep)\right | ) \ dy \underbrace{\leq}_{\mbox{using Korn's inequality}} C ( \int_\Omega | e(\mathbf{w} (\ep ) )|^2 \ dy )^{1/2} , \label{P3}
\end{gather}
\begin{gather}
\left |  \int_{\Sigma^{+}\cup\Sigma^{-}}{ \partial_t \mathcal{P}_{j}}{w_j}(\ep)\ ds\right| =  | \int_{\Omega} \partial_{y_3} ( \frac{\mathcal{P}_{j} |_{y_3 =1} -\mathcal{P}_{j} |_{y_3 =-1}}{2} y_3+\notag \\
\frac{\mathcal{P}_{j} |_{y_3 =1} + \mathcal{P}_{j} |_{y_3 =-1}}{2} ){w_j}(\ep)  dy | \leq  C \int_{\Omega}\left | \partial_{y_3}w_j(\ep)\right |  dy + C \int_{\Omega}\left | {w_j}(\ep) \right |  dy \leq \notag \\
C \left( \int_{\Omega}\left |\nabla  {w_j}(\ep) \right |^2  dy\right)^{1/2}
\leq C\| e({\bf w}(\ep))\|_{L^2(\Omega)} .\label{Phoriz}
\end{gather}
The initial value of ${\bf w(\ep) }$ is calculated using $p_{in}$ and we have ${\bf w(\ep) } |_{t=0} =0.$ Using equation (\ref{ScaledvariationalT2cor}) we find that $\p_t \pi (\ep ) |_{t=0} =0$, as well.
 The integration in time of \eqref{33} gives  estimates (\ref{apriori1T})-(\ref{apriori7T}), but with $L^2$-time norms.

 Next we calculate the time derivative of equations (\ref{ScaledvariationalT1cortab})-(\ref{ScaledvariationalT2cor}), test by $\varphi = \p_{tt} {\bf w(\ep) }  $ and $\zeta = \p_t \pi (\ep)$. After repeating the above calculations, we obtain estimates \eqref{apriori1T}-\eqref{apriori7T}.
\end{proof}
\begin{proposition}\label{Propvkl}
The weak limit $\{{\bf w}^*, \pi^0 \}\in H^1(\Omega\times(0,T))^3 \times H^1(0,T; L^2 (\Omega) )$, $\p_{y_3} \pi^0 \in H^1(0,T; L^2 (\Omega) )$ of the weakly converging subsequence $\{ {\bf w}(\ep), \pi (\ep)\}_{\ep>0}$ belongs to $H^1(0,T; V_{kl}(\Omega))\times  H^1(0,T; L^2 (\Omega) )$, where
\begin{eqnarray*}
&&\hskip-17pt V_{kl}(\Omega)=\{ \mathbf{v} \in  H^1(\Omega)^3, \, e_{j3}(\mathbf{v})=0 \
\mbox{ in } \ \O \mbox{ and }
\mathbf{v} =0 \ \mbox{ on } \partial\omega\times(-1,1) \} . 
\end{eqnarray*}
\end{proposition}
\begin{proof}
Using weak compactness, we find out that there exists a subsequence, denoted by the same subscript $\ep$, and elements ${\bf w}^*\in H^1(\Omega\times(0,T))^3$ and $\pi^0\in H^1(\Omega\times(0,T))$ such that
\begin{eqnarray}
&&{\bf w}(\ep) \rightharpoonup {\bf w}^* \mbox{ weakly in }H^1(\Omega\times(0,T))^3 \mbox{ as } \ep\rightarrow 0, \\
&&\pi (\ep) \rightharpoonup \pi^0 \mbox{ weakly in }H^1(0,T; L^2 (\Omega) ) \mbox{ as } \ep\rightarrow 0, \\
&& \p_{y_3} \pi (\ep) \rightharpoonup \p_{y_3} \pi^0 \mbox{ weakly in }H^1(0,T; L^2 (\Omega) ) \mbox{ as } \ep\rightarrow 0, \\
&&e({\bf w}(\ep)) \rightharpoonup e({\bf w}^*) \mbox{ weakly in }L^2_s(\Omega\times(0,T)) \mbox{ as } \ep\rightarrow 0.
\end{eqnarray}
Using \eqref{apriori4T}-\eqref{apriori5T} we obtain
\begin{eqnarray}
\|e_{j3}({\bf w}(\ep))  \|_{L^2(\Omega\times(0,T)) }&\leq& C\ep, \quad j=1,2,\\
\|e_{33}({\bf w}(\ep))  \|_{L^2(\Omega\times(0,T)) }&\leq& C\ep^2.
\end{eqnarray}
Hence, by weak lower semicontinuity of norm
\begin{eqnarray}
\|e_{j3}({\bf w}^*)  \|_{L^2(\Omega\times(0,T)) }&\leq& \liminf_{\ep\rightarrow 0} \|e_{j3}({\bf w}(\ep))  \|_{L^2(\Omega\times(0,T)) } =0.
\end{eqnarray}
Consequently, ${\bf w}^*\in H^1(0,T;  V_{kl}(\Omega))$.
\end{proof}
\begin{lemma}\label{lemma3} 
The space $V_{kl}$ defined in Proposition \ref{Propvkl} is characterized by the following properties
\begin{equation*}
V_{kl}(\Omega)=\{\mathbf{v} \in  H^1(\Omega)^3, \; v_j=g_j-y_3 \frac{\partial g_3}{\partial y_j}, \; g_j\in H^1_0(\omega), \; \; j=1,2;  \;
 v_3=g_3 \in H_0^2 (\o) \} .
\end{equation*}
\end{lemma}

\begin{proof}
The proof of this Lemma is given by Ciarlet \cite{Ciarlet90}[p.23].
\end{proof}

\begin{corollary}\label{coro1}  There is $\mathbf{w}^0 \in H^1 (0,T; H^1_0 (\o )^3)$,   $ w_3^0 \in H^1(0,T; H^2_0 (\o))$, such that in   ${   \Omega }\times (0,T)$
\begin{equation}\label{charac1}
    w_j^*=w_j^0(y_1,y_2,t)-y_3 \partial_{y_j} w_3^0(y_1,y_2,t), \;\; j=1,2; \;\; w_3^*=w_3^0(y_1,y_2,t).
\end{equation}
\end{corollary}
\begin{proposition}\label{propE0*tabcorr}
Let $< \pi^0 > = \frac{1}{2} \int^{1}_{-1} \pi^0 \ dy_3 $.
Let $$E_{cor}=E_{cor} (y_1,y_2,t)= \frac{\alpha  (1-2\nu ) < \pi^0 >- 2\nu  \mbox{div }_{y_1,y_2}(w_{1}^0,w_{2}^0)}{2(1-\nu )}.$$ Then,
\begin{gather}
\frac{1}{\ep^2}e_{33}({\bf w} (\ep))=\frac{\partial w_{3}(\ep)}{\partial {y_3}}\frac{1}{\ep^2}\rightharpoonup \frac{\alpha (1-2\nu ) \pi^0-2\nu \mbox{ div }_{y_1,y_2}{\bf \tilde w^*} }{2(1-\nu )}=E_{cor}^*=\notag \\
E_{cor}+\frac{\nu y_3}{1-\nu} \Delta_{y_1,y_2} w_{3}^0 + \frac{\alpha  (1-2\nu ) }{2(1-\nu )} (\pi^0 -< \pi^0 >) , \label{Comptab0cor}
\end{gather}
weakly in $L^2(\Omega\times (0,T))$, as $\ep\rightarrow 0$.
\end{proposition}
\begin{proof} After setting $\varphi_j =0$ in equation (\ref{ScaledvariationalT1cortab}) and multiplying by $\ep^2$, we get
 \begin{eqnarray}
&& 2 \int_{\Omega}\sum_{j=1}^{2}   e_{3j}({\bf w} (\ep)) \p_{y_3} \varphi_3 \ dy - \alpha\int_{\Omega}
\pi(\ep)  \partial_{y_3}{ \varphi_3}dy+ \nonumber\\  && \frac{2\nu }{1-2\nu} \int_{\Omega} \mbox{ div}_{y_1,y_2}  {\bf \tilde w} (\ep) \partial_{y_3} \varphi_3 \ dy
+ \frac{2(1-\nu ) }{(1-2\nu )\ep^2} \int_{\Omega} \partial_{y_3}  { w_{3} (\ep) } \partial_{y_3}{  \varphi_3}dy \nonumber\\&
=&
 \ \ep^2 \int_{\Sigma^{+}\cup\Sigma^{-}} {  \mathcal{P}_{3}}{ \varphi_3}ds. \nonumber
\end{eqnarray}
Passing to the limit $\ep \to 0$ yields
\begin{gather*}
    \int^T_0 \int_\Omega \bigg( \frac{2(1-\nu )}{1-2\nu}  E_{cor}^* + \frac{2\nu }{1-2\nu} div_{y_1,y_2}(w_{1}^*,w_{2}^*) - \alpha \pi^0 \bigg) \frac{\p \varphi_3}{\p y_3} \ dy =0
\end{gather*}
for all $\varphi_3  \in C(0,T ; H^1 (\Omega))$, $ \varphi_3 |_{\o}  =0$. Now it is straightforward to conclude (\ref{Comptab0cor}).
\end{proof}

\begin{proposition}\label{propej30} 
It holds
\begin{equation}\label{Shearcon}
    \frac{e_{j3} ( \mathbf{w} (\ep) ) }{\ep} \rightharpoonup 0, \quad \mbox{ weakly in } \; L^2(\Omega\times (0,T)), \quad \mbox{as } \; \ep\rightarrow 0, \ j=1,2.
\end{equation}
\end{proposition}
\begin{proof}
Setting $\varphi_3=0$ in \eqref{ScaledvariationalT1cortab} and multiplying by ${\ep}$ yield
\begin{eqnarray*}
&&\frac{2}{\ep}\int_0^T\int_{\Omega}\sum_{j=1}^2 e_{3j}({\bf w}(\ep))\partial_{y_3}\varphi_j+\frac{2\nu }{(1-2\nu)\ep} \int_0^T\int_{\Omega}\partial_{y_3} w_3(\ep) \mbox{div}_{y_1,y_2}\tilde{\bf \varphi} =O(\ep).
\end{eqnarray*}
In the limit $\ep \rightarrow 0$ we obtain
\begin{equation}
    \frac{e_{j3} ( \mathbf{w} (\ep) ) }{\ep} \rightharpoonup \chi_j, \quad \mbox{ weakly in } \; L^2(\Omega\times (0,T))
\end{equation}
and
$\int_{\Omega}\chi_j \partial_{y_3} \varphi_j =0$, $\forall \varphi_j\in H^1(\Omega),$ $\varphi_j|_{\partial \omega \times (-1,1)}=0.$
 It finishes the proof of Proposition \ref{propej30}.
\end{proof}

\begin{proposition}\label{strechlim1}
$(w_{1}^0 , w_{2}^0 ) $ satisfies  system (\ref{StrechJ01Tcor})-(\ref{StrechJ04Tcor}).
\end{proposition}
\begin{proof}
We test equation (\ref{ScaledvariationalT1cortab}) by $\varphi \in V_{KL} (\Omega) $. Since $e_{j3} (\varphi ) =0$, we have
\begin{gather} 2 \int_{\Omega}\sum_{i,j=1}^{2}   e_{ij}({\bf w} (\ep)) e_{ij}({\bf \varphi}) dy + \frac{2\nu }{1-2\nu} \int_{\Omega} \mbox{ div}_{y_1,y_2}  {\bf \tilde w} (\ep) \mbox{ div}_{y_1,y_2}{\bf \tilde  \varphi} \ dy \notag \\
-\alpha\int_{\Omega}
\pi(\ep)  \mbox{ div}_{y_1,y_2}{\bf \tilde  \varphi} \ dy + \frac{2\nu }{(1-2\nu )\ep^2} \int_{\Omega} \partial_{y_3}  {w_{3} (\ep) } \mbox{ div}_{y_1,y_2}{\bf \tilde  \varphi} \ dy \notag \\
=  \sum_{j=1}^{2}  \int_{\Sigma^{+}\cup\Sigma^{-}}{  \mathcal{P}_{j}}{ \varphi_j} \ ds { + \int_{\Sigma^{+}\cup\Sigma^{-}}{  \mathcal{P}_{3}}{ \varphi_3} \ ds .} \label{LimitVKL1cor}
\end{gather}
We use Proposition \ref{propE0*tabcorr} and pass to the limit $\ep \to 0$ in equation (\ref{LimitVKL1cor}). It yields

\begin{gather}
    \int_{\Omega} \big( 2 \sum_{i,j=1}^{2}   e_{ij}({\bf w}^*) e_{ij}({\bf \varphi}) -\frac{1-2\nu}{1-\nu} (\alpha  \pi^0  -\notag \\ \frac{2\nu}{1-2\nu}\mbox{ div}_{y_1,y_2}  {\bf \tilde w}^* ) \mbox{ div}_{y_1,y_2}{\bf \tilde  \varphi} \big) dy =\sum_{j=1}^{2}  \int_{\Sigma^{+}\cup\Sigma^{-}}{  \mathcal{P}_{j}}{ \varphi_j} \ ds { + \int_{\Sigma^{+}\cup\Sigma^{-}}{  \mathcal{P}_{3}}{ \varphi_3} \ ds } \label{LimitVKL02Dcor}
\end{gather}
and choice $\varphi = (g_1 , g_2 , 0)$, $g_j \in H^1_0 (\o)$, gives equation (\ref{StrechJ04Tcor}).\vskip1pt

Finally, we take $\zeta =\zeta (y_1 , y_2 ,t)$ as a test function in (\ref{ScaledvariationalT2cor}). Using Proposition
\ref{propE0*tabcorr} and zero initial data, we obtain the equality (\ref{StrechJ01Tcor}).
\end{proof}

\begin{proposition}
The pressure equation reads
\begin{gather}
    \p_t \big\{ (\gamma + \alpha^2 \frac{1-2\nu}{2(1-\nu)} ) \pi^0 + \alpha \frac{1-2\nu}{1-\nu} \mbox{div }_{y_1 , y_2 } ( w^0_{1} , w^0_{2} ) \big\} -\frac{\p^2  \pi^0}{\p y_3^2} -\notag \\
     \alpha \frac{1-2\nu}{1-\nu} y_3 \Delta _{y_1 , y_2 } \p_t  w^0_{3} =0 \quad \mbox{ in } \quad \Omega \times (0,T) ,  \label{presstab1cor} \\
    \p_{y_3}  \pi^0 |_{y_3 =1} = \p_{y_3}  \pi^0 |_{y_3 =-1} =-U^1 (y_1 , y_2 , t) \quad \mbox{ in } \quad  (0,T) ,  \label{presstab2cor} \\
    \pi^0 |_{t=0} =0 \quad \mbox{ in } \quad  \Omega .  \label{presstab3cor}
\end{gather}
\end{proposition}
\begin{proof} Passing to the limit $\ep \to 0$ in equation (\ref{ScaledvariationalT2cor}) yields
    \begin{eqnarray}
&&\gamma \int_{\Omega} \partial_t \pi^0 \zeta \ dy + \int_{\Omega} \alpha (\mbox{ div}_{y_1,y_2}  {\partial_t \bf \tilde w}^*  +\p_t E^*_{cor} ) \zeta dy
\nonumber\\&& + \int_{\Omega} \frac{\partial {  \pi^0 }}{\partial y_3}\frac{\partial \zeta}{\partial y_3} \ dy = -\int_{\Sigma^+ \cup \Sigma^-} \pm U^1 \zeta ds , \quad \forall {\zeta}\in H^1(\Omega).\label{ScaledvariationalPtab1cor}
\end{eqnarray}
Using Proposition
\ref{propE0*tabcorr} and zero initial data, we obtain from (\ref{ScaledvariationalPtab1cor}) system (\ref{presstab1cor})-(\ref{presstab3cor}).
\end{proof}

\begin{corollary} The function $\pi_w= \pi^0 -< \pi^0 >$ satisfies  system (\ref{presstabD1cor})-(\ref{presstabD3cor}).
\end{corollary}
\begin{proposition} The limit plate deflection  $w^0_{3}$ satisfies  equation (\ref{Bendtab0cor1}).
\end{proposition}
\begin{proof} We take as test $\varphi =\displaystyle (-y_3 \frac{\p g_3 }{\p y_1} , -y_3 \frac{\p g_3 }{\p y_2} , g_3 )$, with $g_3 \in H_0^2 (\o )$. It yields
\begin{gather*}
    e_{11} (\varphi ) = -y_3 \frac{\p^2 g_3 }{\p y_1^2} , \quad e_{22} (\varphi ) = -y_3 \frac{\p^2 g_3 }{\p y_2^2} , \quad e_{12} (\varphi ) = -y_3 \frac{\p^2 g_3 }{\p y_1 \p y_2} , \\
    \mbox{ div }_{y_1 , y_2} (\varphi_1 , \varphi_2 ) = -y_3 \Delta_{y_1 , y_2 } g_3
\end{gather*}
Passing to the limit  $\ep \to 0$ in equation (\ref{ScaledvariationalT1cortab}) yields
\begin{gather}
    \int_{\Omega} \big( 2 y_3^2 \sum_{i,j=1}^{2}   \frac{\p^2 { w}_{3}^0}{\p y_i \p y_j}  \frac{\p^2 g_{3}}{\p y_i \p y_j} + y_3 \frac{1-2\nu}{1-\nu} (\alpha \pi^0 +
     \frac{2\nu y_3}{1-2\nu}\Delta_{y_1,y_2}  { w}_{3}^0 ) \Delta_{y_1,y_2} g_3 \big) dy \notag \\
     =-\sum_{j=1}^{2}  \int_{\Sigma^{+}\cup\Sigma^{-}}{  \mathcal{P}_{j}}{ y_3 \frac{\p g_3 }{\p y_j}}ds+  \int_{\Sigma^{+}\cup\Sigma^{-}} {  \mathcal{P}_{3}}{ g_3}ds .\label{LimitVKL0Bcor1}
\end{gather}
Equation (\ref{LimitVKL0Bcor1}) implies (\ref{Bendtab0cor1}).
\end{proof}
\begin{proposition}\label{Wpbending} System (\ref{presstabD1cor})-(\ref{Bendtab0cor1}) has a unique solution.\end{proposition}

\begin{proof} It is enough to study the problem with the homogeneous data. Let $(u^0_3 , p^0 )\in H^1 (0,T; H^2_0 (\o))\times H^1 (0,T; H^1 (\Omega))$, $<p^0> =0$, be a solution. We take $g_3 = \p_t  u_{3}^0$ in (\ref{LimitVKL0Bcor1}) and use $\zeta = p^0  $ as a test function for system (\ref{presstabD1cor})-(\ref{presstabD3cor}). After summing up we obtain the equality
\begin{gather}
    \p_t \int_\Omega \frac{y_3^2}{1-\nu} | \Delta_{y_1 , y_2} u_{3}^0 |^2 \ dy + \int_\Omega \frac{y_3 \alpha (1-2\nu)}{1-\nu} p^0 \Delta_{y_1 , y_2} \p_t  u_{3}^0 \ dy  +\int_\Omega | \p_{y_3}   p^0   |^2 \ dy +\notag \\
    \frac{1}{2} \p_t \int_\Omega  (\gamma + \alpha^2 \frac{1-2\nu}{2(1-\nu)} )   ( p^0  )^2 \ dy   -
     \int_\Omega \frac{y_3 \alpha (1-2\nu)}{1-\nu}  p^0  \Delta_{y_1 , y_2} \p_t  u_{3}^0 \ dy =0. \label{Bendtaben0}
\end{gather}
{  After} noting that the second and the fifth terms cancel each other, we conclude that $p^0 =0$ and
$\Delta_{y_1,y_2} u_{3}^0=0$ in $\o$. The last equation, together with $ u_{3}^0 \in H^2_0 (\o)$, yields $u_{3}^0=0$.
\end{proof}

 \begin{proposition}\label{prop1Tcor11} The whole sequence $ \{ {\bf w}(\ep) , p (\ep ) \} $ satisfies
\begin{gather}
{ w}_{j}(\ep)  \rightharpoonup  { w}_{j}^0 -y_3 \frac{\p { w}_{3}^0}{\p y_j}, \;\; j=1,2, \; 
   \mbox{ weakly in }H^1(\Omega\times(0,T)) \mbox{ as } \ep\rightarrow 0,\label{limitstab1defcor} \\
  { w}_{3}(\ep)    \rightharpoonup { w}_{3}^0
   \mbox{ weakly in }H^1(\Omega\times(0,T)) \mbox{ as } \ep\rightarrow 0,\label{limitstab1Bdefcor} \\
 \pi (\ep)  \rightharpoonup \pi^0 = \pi_m + \pi_w \mbox{ weakly in } \quad H^1(0,T; L^2(\Omega)) \quad \mbox{ as } \ep\rightarrow 0, \label{limitstab2defcor}\\
\p_{y_3} \pi (\ep)    \rightharpoonup \p_{y_3} \pi^0 \mbox{ weakly in } \quad H^1(0,T; L^2(\Omega)) \quad \mbox{ as } \ep\rightarrow 0, \label{limitstab3defcor}
\end{gather}
where
  $\mathbf{w}^0$ and $ \pi_m = < \pi^0 >$ are given by by (\ref{StrechJ01Tcor})-(\ref{StrechJ04Tcor}) and $ \{ w^0_{3}, \pi_w  \} $ by (\ref{presstabD1cor})- (\ref{Bendtab0cor1}).
\end{proposition}

\section{Strong convergence}\label{strongcoup}

In this section we establish that the weak convergences from previous section imply the strong convergences.

With such aim, we introduce the corrected unknowns, for which the weak convergence to zero was already established in Subsection    \ref{convg6}.

Let $\Psi_\ep \in C^\infty_0 (\o )$ be a regularized truncation  of the indicator function $1_{\o} $, equal to $1_{\o} $ if dist $(y, \p \o ) \geq {\ep}$ and such that $\displaystyle || \nabla _{y_1 , y_2 } \Psi_\ep ||_{L^q (\o )} =C \ep^{1/q -1}$ . We set
\begin{equation}\label{Corrtab0S}
    \left\{
      \begin{array}{ll}
        \displaystyle \xi_{j} (\ep) = w_{j} (\ep ) - w_{j}^0  +y_3 \frac{\p w^0_3}{\p y_j}  , \quad j=1,2; &  \\
       \displaystyle \xi_{3} (\ep ) = w_{3} (\ep) - w_3^0 - \ep^2 \mathcal{E}  ,   \ &  \\
        \displaystyle \kappa (\ep ) =\pi (\ep ) - \pi^0 , &
      \end{array}
    \right.
\end{equation}
with $\displaystyle \mathcal{E} = \Psi_\ep (y_1 , y_2 ) \int^{y_3}_0 E_{cor}^* (y_1, y_2, a , t) \ da .$
The choice of correcting terms is explained by the following result
\begin{lemma}\label{Convsd1} We have
\begin{gather}
    \frac{1}{\ep} e_{j3} (\xi (\ep )) \rightharpoonup 0 \quad \mbox{weakly in} \quad H^1 (0,T; L^2(\Omega)), \;  j=1,2,3, \; \mbox{ as } \; \ep \to 0, \label{Shearco1} \\
    \xi (\ep ) \rightharpoonup 0 \quad \mbox{weakly in} \quad H^1 (0,T; H^1_0 (\Omega) )^3, \; \mbox{ as } \; \ep \to 0, \label{Shearco2} \\
    \frac{1}{\ep^2} e_{33} (\xi (\ep )) \rightharpoonup 0 \quad \mbox{weakly in} \quad H^1 (0,T; L^2(\Omega)), \;   \mbox{ as } \; \ep \to 0, \label{Shearco3}
\end{gather}
\end{lemma}
\begin{proof}
First, using definition (\ref{Corrtab0S}) we find out that
\begin{gather*}
     \frac{1}{\ep} e_{j3} (\xi (\ep )) = \frac{1}{\ep} e_{j3} (\mathbf{w} (\ep ))  -\frac{\ep}{2} \frac{\p \mathcal{E}}{\p y_j} \underbrace{\rightharpoonup}_{by (\ref{Shearcon})} 0,
\end{gather*}
implying (\ref{Shearco1}).\vskip1pt
(\ref{Shearco2}) follows from Proposition \ref{Propvkl} and Corollary \ref{coro1} and smallness of the correcting terms.\vskip1pt
Finally,
\begin{gather*}
     \frac{1}{\ep^2} e_{33} (\xi (\ep )) = \frac{1}{\ep^2} e_{33} (\mathbf{w} (\ep )) - \frac{\p \mathcal{E}}{\p y_3} \underbrace{\rightharpoonup}_{by (\ref{Comptab0cor})} 0.
\end{gather*}
\end{proof}
Next, using equation (\ref{LimitVKL02Dcor}) we obtain
\begin{gather}
2 \int_{\Omega}\sum_{i,j=1}^{2}   e_{ij}({\bf \tilde w}^0 -y_3 \nabla_{y_1 , y_2 } w_3^0 ) e_{ij}({\bf \varphi}) dy + \frac{2\nu }{1-2\nu} \int_{\Omega} \mbox{ div}_{y_1,y_2}  ({\bf \tilde w}^0 \notag\\
-y_3 \nabla_{y_1 , y_2 } w_3^0  ) \mbox{ div}_{y_1,y_2}{\bf \tilde  \varphi}dy
-\alpha\int_{\Omega}
\pi^0  \mbox{ div}_{y_1,y_2}{\bf \tilde  \varphi}dy  + 2 \int_{\Omega}\sum_{j=1}^{2}  \frac{\p \mathcal{E}}{\p y_j}  e_{3j}({\bf \varphi}) dy\notag\\
+\frac{1}{\ep^2}\Big\{
 - \alpha\int_{\Omega}
\pi^0  \partial_{y_3}{ \varphi_3}dy+
 \frac{2\nu }{1-2\nu} \int_{\Omega} \mbox{ div}_{y_1,y_2}  ({\bf \tilde w}^0
-y_3 \nabla_{y_1 , y_2 } w_3^0  ) \partial_{y_3}{ \varphi_3}dy\notag \\
+ \frac{2\nu }{1-2\nu} \int_{\Omega}   \frac{\p ( w_3^0 + \ep^2 \mathcal{E} ) }{\p y_3} \mbox{ div}_{y_1,y_2}{\bf \tilde  \varphi}dy\Big\}+ \frac{2(1-\nu ) }{(1-2\nu )\ep^4} \int_{\Omega} \partial_{y_3}  ( w_3^0 + \ep^2 \mathcal{E} ) \partial_{y_3}{ \varphi_3}dy\notag\\
=  \sum_{j=1}^{2}  \int_{\Sigma^{+}\cup\Sigma^{-}}{  \mathcal{P}_{j}}{ \varphi_j}ds+  \int_{\Sigma^{+}\cup\Sigma^{-}} {  \mathcal{P}_{3}}{ \varphi_3}ds +\frac{2\nu }{1-2\nu} \int_{\Omega} (\Psi_\ep -1 )E^*_{cor} \mbox{ div}_{y_1,y_2}{\bf \tilde  \varphi} \ dy  \notag \end{gather}\begin{gather}
    + 2 \int_\Omega \sum_{j=1}^2 \frac{\p \mathcal{E}}{\p y_j} e_{j3} (\varphi ) \ dy  -\frac{\alpha}{\ep^2}\int_{\Omega}
\pi^0   \partial_{y_3}{ \varphi_3}\ dy
 +\frac{1}{\ep^2}\frac{2\nu }{1-2\nu} \int_{\Omega} \mbox{ div}_{y_1,y_2}  ({\bf \tilde w}^0
-\notag \\ y_3 \nabla_{y_1 , y_2 } w_3^0  ) \partial_{y_3}{ \varphi_3}dy
 + \frac{2(1-\nu ) }{(1-2\nu )\ep^2} \int_{\Omega}\Psi_\ep E^*_{cor}  \partial_{y_3}{ \varphi_3}dy, \quad \forall {\bf \varphi}\in V,  t\in (0,T).\label{ScaledvariationalSaux1}\end{gather}
Consequently, we observe that  $\{ \xi (\ep ) , \kappa (\ep) \}$ satisfy the following variational equation
\begin{eqnarray}
&&2 \int_{\Omega}\sum_{i,j=1}^{2}   e_{ij}(\xi (\ep)) e_{ij}({\bf \varphi}) dy + \frac{2\nu }{1-2\nu} \int_{\Omega} \mbox{ div}_{y_1,y_2}  { \tilde \xi} (\ep) \mbox{ div}_{y_1,y_2}{\bf \tilde  \varphi}dy \nonumber\\&&-\alpha\int_{\Omega}
\kappa (\ep)  \mbox{ div}_{y_1,y_2}{\bf \tilde  \varphi}dy
+\frac{1}{\ep^2}\Big\{4 \int_{\Omega}\sum_{j=1}^{2}   e_{3j}(\xi (\ep)) e_{3j}({\bf \varphi}) dy \nonumber\\
&& - \alpha\int_{\Omega}
\kappa(\ep)  \partial_{y_3}{ \varphi_3}dy+ \frac{2\nu }{1-2\nu} \int_{\Omega} \mbox{ div}_{y_1,y_2}  { \tilde \xi} (\ep) \partial_{y_3}{ \varphi_3}dy\nonumber\\
&&+ \frac{2\nu }{1-2\nu} \int_{\Omega} \partial_{y_3}  {\xi_{3} (\ep) } \mbox{ div}_{y_1,y_2}{\bf \tilde  \varphi}dy\Big\}+ \frac{2(1-\nu ) }{(1-2\nu )\ep^4} \int_{\Omega} \partial_{y_3}  { \xi_{3} (\ep) } \partial_{y_3}{  \varphi_3}dy = \nonumber\\
&& \hskip-17pt -2 \int_\Omega  \sum_{j=1}^2  \frac{\p \mathcal{E}}{\p y_j} e_{j3} (\varphi ) \ dy    +\frac{\alpha}{\ep^2}\int_{\Omega}
\pi^0   \partial_{y_3}{ \varphi_3}\ dy - \frac{2\nu }{1-2\nu} \int_{\Omega} (\Psi_\ep -1 )E^*_{cor} \mbox{ div}_{y_1,y_2}{\bf \tilde  \varphi} \ dy \nonumber \\
&&-\frac{2\nu }{1-2\nu} \int_{\Omega} \mbox{ div}_{y_1,y_2}  ({\bf \tilde w}^0
-y_3 \nabla_{y_1 , y_2 } w_3^0  ) \frac{\partial_{y_3}{ \varphi_3}}{\ep^2} \ dy\nonumber \\
 && - \frac{2(1-\nu ) }{(1-2\nu )\ep^2} \int_{\Omega}\Psi_\ep  E^*_{cor}  \partial_{y_3}{ \varphi_3}dy, \quad
 \forall {\bf \varphi}\in V,  t\in (0,T).\label{ScaledvariationalS1}
\end{eqnarray}
Effective pressure equation (\ref{ScaledvariationalPtab1cor}) implies
\begin{eqnarray}
&&\gamma \int_{\Omega} \partial_t \pi^0 \zeta \ dy + \int_{\Omega} \alpha (\mbox{ div}_{y_1,y_2}  \partial_t  {\bf \tilde w}^*     +\ep^{-2} \p_t \p_{y_3} (w_3^0 +\ep^2 \mathcal{E} )) \zeta dy
\nonumber\\&& + \int_{\Omega} \frac{\partial \pi^0}{\partial y_3}\frac{\partial \zeta}{\partial y_3} \ dy + \ep^2 \int_{\Omega} \nabla_{y_1 , y_2 } \pi^0 \nabla_{y_1 , y_2 } \zeta \ dy = \ep^2 \int_{\Omega} \nabla_{y_1 , y_2 } \pi^0 \nabla_{y_1 , y_2 } \zeta \ dy     \nonumber\\&& -\int_{\Sigma^+ \cup \Sigma^-} \pm U^1 \zeta ds +
\alpha \int_{\Omega} (\Psi_\ep -1 ) \p_t E^*_{cor} \zeta \ dy
, \quad \forall {\zeta}\in H^1(\Omega).\label{ScaledvariationalPStab1cor}
\end{eqnarray}
Consequently the pressure equation for
$\{ \xi (\ep ) , \kappa (\ep) \}$ is
\begin{eqnarray}
&&\gamma \int_{\Omega} \partial_t \kappa (\ep) \zeta dy + \int_{\Omega} \alpha \mbox{ div}_{y_1,y_2}  {\partial_t  \tilde \xi} (\ep) \zeta dy +\ep^2 \int_{\Omega} \nabla_{y_1,y_2}  \kappa(\ep)  \nabla_{y_1,y_2} \zeta dy\nonumber\\&&+\alpha  \ep^{-2}  \int_{\Omega} \partial_{y_3}  {\partial_t \xi_{3} (\ep) } \zeta \ dy + \int_{\Omega} \frac{\partial \kappa (\ep)}{\partial y_3}\frac{\partial \zeta}{\partial y_3} \ dy=- \ep \int_{-1}^1 \int_{\partial\omega} V\zeta ds\nonumber\\
&&  -\ep^2 \int_{\Omega} \nabla_{y_1 , y_2 } \pi^0 \nabla_{y_1 , y_2 } \zeta \ dy - \alpha \int_{\Omega} (\Psi_\ep -1 ) \p_t E^*_{cor} \zeta \ dy , \quad \forall {\zeta}\in H^1(\Omega).\label{ScaledvariationalTS2cor}
\end{eqnarray}
 By the choice of the correcting terms, $\xi (\ep ) =0$ on $\p \o \times (-1,1).$ Now we test (\ref{ScaledvariationalS1}) by $\varphi= \p_t \xi (\ep)$, (\ref{ScaledvariationalTS2cor}) by $\zeta = \kappa (\ep)$ and add the obtained equalities, to obtain the following energy equality
\begin{eqnarray}\label{33tabcor12}
&& \hskip-17pt \frac{1}{2} \frac{d}{dt} \Big\{ 2 \int_{\Omega}\sum_{i,j=1}^{2}   \left |e_{ij}({\xi} (\ep) )\right|^2 dy + \frac{2\nu }{1-2\nu} \int_{\Omega} \left | \mbox{ div}_{y_1,y_2}  { \tilde \xi} (\ep)  \right |^2 \ dy + \gamma \int_{\Omega} \left | \kappa (\ep)\right|^2 \ dy \nonumber\\
&&\hskip-17pt +\frac{4}{\ep^2}  \int_{\Omega}\sum_{j=1}^{2}  \left | e_{3j}({\xi} (\ep) )\right |^2 \ dy +  \frac{4\nu }{(1-2\nu )\ep^2} \int_{\Omega} \mbox{ div}_{y_1,y_2}  { \tilde \xi} (\ep)  \partial_{y_3}{\xi_{3} (\ep)} \ dy \nonumber\\
&&\hskip-17pt +  \frac{2 (1-\nu )}{(1-2\nu )\ep^4} \int_{\Omega}\left | \partial_{y_3}  \xi_{3} (\ep)\right|^2 dy \Big\}+ \ep^2 \int_{\Omega} \left | \nabla_{y_1,y_2}  {\kappa (\ep) } \right |^2dy + \int_{\Omega} \left | \frac{\partial \kappa (\ep) }{\partial y_3}  \right |^2 \ dy\nonumber\\
&&\hskip-17pt =\frac{d}{dt} \Big\{ -2 \int_\Omega  \sum_{j=1}^2 \ep \frac{\p \mathcal{E}}{\p y_j} \frac{e_{j3} (\xi (\ep) )}{\ep} \ dy    +\frac{\alpha}{\ep^2}\int_{\Omega}
\pi^0   \partial_{y_3}{ \xi_3 (\ep)}\ dy -\nonumber \\&&\hskip-17pt
\frac{2\nu }{1-2\nu} \int_{\Omega} \bigg( \mbox{ div}_{y_1,y_2}  ({\bf \tilde w}^0
-y_3 \nabla_{y_1 , y_2 } w_3^0  ) \frac{\partial_{y_3}{ \xi_3 (\ep) }}{\ep^2}
+ (\Psi_\ep -1 )E^*_{cor} \mbox{ div}_{y_1,y_2}{\bf \tilde  \xi (\ep) } \bigg) \ dy
\nonumber \\
 &&\hskip-17pt - \frac{2(1-\nu ) }{(1-2\nu )\ep^2} \int_{\Omega} \Psi_\ep E^*_{cor}  \partial_{y_3}{ \xi_3 (\ep) }dy \Big\}
 +2 \ep \int_\Omega  \sum_{j=1}^2   \frac{\p^2 \mathcal{E}}{\p y_j \p t}  \frac{e_{j3} (\xi (\ep) )}{\ep} \ dy \nonumber \\
 &&\hskip-17pt -\frac{\alpha}{\ep^2}\int_{\Omega}
\p_t \pi^0   \partial_{y_3}{ \xi_3 (\ep)}\ dy + \frac{2(1-\nu ) }{(1-2\nu )\ep^2} \int_{\Omega} \Psi_\ep \p_t E^*_{cor}  \partial_{y_3}{ \xi_3 (\ep) }dy + \nonumber \\&&
 \frac{2\nu }{1-2\nu} \int_{\Omega} \bigg( \p_t \mbox{ div}_{y_1,y_2}  ({\bf \tilde w}^0
-y_3 \nabla_{y_1 , y_2 } w_3^0  ) \frac{\partial_{y_3}{ \xi_3 (\ep) }}{\ep^2} +
(\Psi_\ep -1 )\p_t E^*_{cor} \mbox{ div}_{y_1,y_2}{\bf \tilde  \xi (\ep) } \bigg)
\ dy\nonumber \\
&&\hskip-17pt
 -\ep^2 \int_{\Omega} \nabla_{y_1 , y_2 } \pi^0 \nabla_{y_1 , y_2 } \kappa (\ep) \ dy - \ep \int_{-1}^1 \int_{\partial\omega} V \kappa (\ep) ds - \alpha \int_{\Omega} (\Psi_\ep -1 ) \p_t E^*_{cor} \kappa (\ep) \ dy .
\end{eqnarray}
\begin{proposition}\label{proofstrong}  For the whole sequence $\{ \xi (\ep ), \kappa (\ep) \}$ we have
\begin{gather}
e_{ij} ( \xi (\ep ) ) \to 0 \quad \mbox{strongly in} \quad H^1 (0,T; L^2 (\Omega )), \, i,j=1,2, \; \mbox{ as } \; \ep\to 0, \label{St1} \\
    \xi (\ep )  \to 0 \quad \mbox{strongly in} \quad H^1 (0,T; L^2 (\Omega ))^3, \; \mbox{ as } \; \ep\to 0, \label{St2} \\
    \frac{e_{i3} ( \xi (\ep ) )}{\ep} \to 0 \quad \mbox{strongly in} \quad H^1 (0,T; L^2 (\Omega )),  \; \mbox{ as } \; \ep\to 0, \label{St3} \\
     \frac{\p  \xi_3 (\ep ) }{\ep^2 \p y_3} \to 0 \quad \mbox{strongly in} \quad H^1 (0,T; L^2 (\Omega )),  \; \mbox{ as } \; \ep\to 0, \label{St4} \\
     \kappa (\ep ) \to 0 \quad \mbox{strongly in} \quad H^1 (0,T; L^2 (\Omega )), \; \mbox{ as } \; \ep\to 0, \label{St5} \\
     \p_{y_3} \kappa (\ep ) \to 0 \quad \mbox{strongly in} \quad H^1 (0,T; L^2 (\Omega )), \; \mbox{ as } \; \ep\to 0, \label{St6} \\
     \ep \nabla_{y_1 , y_2 } \kappa (\ep ) \to 0 \quad \mbox{strongly in} \quad H^1 (0,T; L^2 (\Omega )), \; \mbox{ as } \; \ep\to 0. \label{St60}
\end{gather}
\end{proposition}
\begin{proof} We follow proof of Proposition \ref{prop1T} and use equality (\ref{33tabcor12}). In fact only new term to be estimated is $\displaystyle \ep \int_{-1}^1 \int_{\partial\omega} V \kappa (\ep) ds$.  We recall the well-known interpolation inequality
$$ || \zeta ||_{L^2 (\p \omega \times (-1,1)) } \leq C || \zeta ||_{L^2 ( \Omega)}^{1/2}  || \zeta ||_{L^2 (-1,1 ; H^1 ( \omega) )}^{1/2} .$$
It yields
\begin{gather}
| \ep \int^1_{-1} \int_{\p \o} V \kappa (\ep ) \ ds | \leq C \ep || V ||_{L^2 (\p \omega \times (-1,1)) } || \kappa (\ep ) ||_{L^2 (\Omega )}^{1/2} \big(  || \kappa (\ep ) ||_{L^2 (\Omega )}^{1/2} + \notag \\
|| \nabla_{y_1 , y_2} \kappa (\ep ) ||_{L^2 (\Omega )^2}^{1/2} \big) \leq C_1 \ep^{2/3} || \kappa (\ep ) ||_{L^2 (\Omega )} + \frac{\ep^2}{2} || \nabla_{y_1 , y_2} \kappa (\ep ) ||_{L^2 (\Omega )^2}^2.
\label{estV}
\end{gather}
Now we integrate equality (\ref{33tabcor12}) in time, use Proposition \ref{prop1Tcor11} and Lemma \ref{Convsd1} and conclude the strong convergence in $L^2 (\Omega \times (0,T))$. Iterating the argument after calculating the time derivative of equations (\ref{ScaledvariationalS1}) and (\ref{ScaledvariationalTS2cor}), yields the convergences (\ref{St1})-(\ref{St60}).
\end{proof}
\section{Convergence of the poroelastic stress}\label{strongstressco}
The rescaled stress $\sigma ({\bf w}(\ep))$ is, in analogy with (\ref{Sig}), given by
\begin{equation}\label{Sigres}
    \sigma ({\bf w}(\ep)) = 2 e({\bf w}(\ep ) ) + ( \frac{2\nu }{1-2\nu} \mbox{ div }{\bf w}(\ep) - \alpha \pi (\ep) ) I  \; \mbox{ in } \; \Omega \times (0,T).
\end{equation}
Nevertheless, this quantity does not correspond to the rescaled dimensionless physical stress
$\sigma^\ep $. We introduce the rescaled poroelastic stress $\sigma (\ep)$ by
\begin{equation}\label{Signew}
   \frac{\sigma (\ep)}{\ep}= \begin{bmatrix}
      \displaystyle 2 \frac{\p w_1 (\ep)}{\p y_1} +D(\ep)   &  \displaystyle 2 e_{12} ({\bf w}(\ep))  & \displaystyle \frac{1}{\ep} e_{13} ({\bf w}(\ep)) \\
      \displaystyle 2 e_{12} ({\bf w}(\ep)) & \displaystyle 2 \frac{\p w_2 (\ep)}{\p y_2} + D(\ep) &  \displaystyle  \frac{1}{\ep} e_{23} ({\bf w}(\ep))\\
      \displaystyle  \frac{1}{\ep} e_{13} ({\bf w}(\ep))  & \displaystyle  \frac{1}{\ep} e_{23} ({\bf w}(\ep)) & \displaystyle \frac{2}{\ep^2} \frac{\p w_3 (\ep)}{\p y_3} + D(\ep)  \\
    \end{bmatrix} ,
\end{equation}
where
\begin{equation}\label{Diag}
    D(\ep) = \frac{2\nu}{1-2\nu} \mbox{div}_{y_1 , y_2} {\bf \tilde w} (\ep) -\alpha \pi (\ep ) + \frac{2\nu}{1-2\nu} \frac{1}{\ep^2} \frac{\p w_3 (\ep)}{\p y_3} .
\end{equation}

As a direct consequence of Proposition \ref{proofstrong}, we obtain the following convergences for the stresses
\begin{proposition}\label{stresscon} Let   $ \{ {\bf \tilde w}^0 ,  \pi_m \} $ be given   by (\ref{StrechJ01Tcor})-(\ref{StrechJ04Tcor}) and $\{ w_3^0 , \pi_w \} $ by (\ref{presstabD1cor})- (\ref{Bendtab0cor1}). Let $\displaystyle E^*_{cor}$ be defined by (\ref{Comptab0cor}).
Then we have
\begin{gather}
 D(\ep)   \to - 2 E^*_{cor} 
 \ \mbox{ in } \quad H^1 (0,T; L^2 (\Omega)) \quad \mbox{as } \; \ep \to 0,
\label{D01}  \\
\frac{\sigma_{jj} (\ep )  }{\ep} -2 e_{jj} (\mathbf{w}^0 ) +2y_3 \frac{\p^2 w^0_3 }{\p y_j^2} + 2 E^*_{cor}
 \to 0 \notag \\
 \quad \mbox{ in } \quad H^1 (0,T; L^2 (\Omega)) \quad \mbox{as } \; \ep \to 0, \; j=1,2,
\label{D02} \\ 
\hskip-12pt \frac{\sigma_{12} (\ep )  }{\ep} -2 e_{12} (\mathbf{w}^0 ) +2y_3 \frac{\p^2 w^0_3 }{\p y_1 \p y_2} \to 0 \, \mbox{ in } \, H^1 (0,T; L^2 (\Omega)) \, \mbox{ as } \ \ep \to 0,
\label{D03} \\
\frac{\sigma_{j3} (\ep )}{\ep}  \to 0 \, \mbox{ in } \, H^1 (0,T; L^2 (\Omega)) \, \mbox{ as } \ \ep \to 0, \; {  j=1,2,}
\label{D04} \\
\frac{\sigma_{33} (\ep )}{\ep} \to 0 \, \mbox{ in } \, H^1 (0,T; L^2 (\Omega)) \, \mbox{ as } \ \ep \to 0.
\label{D05}
\end{gather}
\end{proposition}
\section{Appendix: A classical Kirchhoff type plate equations derivation}\label{append}

We follow \cite{fung} and, together with Kirchhoff's hypothesis, suppose that
\begin{enumerate}
  \item The vertical deflection of the plate takes form $u_3=w=w(x_1 , x_2 ,t)$.
  \item For small deflections $w$ and rotations $(\vartheta_1 , \vartheta_2) $ we have
  $\displaystyle  \vartheta_1 =\frac{\p w}{\p x_2}$ and  $\displaystyle  \vartheta_2 =-\frac{\p w}{\p x_1} .$ Then, Kirchhoff's hypothesis implies that
  the displacements $\mathbf{u}$ satisfy
   \begin{equation}\label{HypoKirchh} u_1 = u^\omega_1 (x_1 , x_2 ,t) - x_3 \frac{\partial w}{\partial x_1}; \, u_2 = u^\omega_2 (x_1 , x_2 ,t) - x_3 \frac{\partial w}{\partial x_2} ,\end{equation}
  where $(u^{\o}_1,u^{\o}_2) $ are the tangential displacements of points lying on the midsurface.
   Now a direct calculation gives  $e_{j3} (\mathbf{u}) =0, \ j=1,2$. Since
  $\displaystyle \sigma = 2G e({\bf u} ) + ( \frac{2\nu G}{1-2\nu} \mbox{ div }{\bf u}- \alpha p ) I,$
  we conclude that also $\sigma_{j3} =0 , \ j=1,2$.
  \item For the pressure field, we impose the normal velocities $U^{\ell }$ at the top and bottom surfaces and suppose (\ref{pressureapp}).
 Finally,  $\sigma_{33}$ is also supposed small throughout the plate. {  The later }  assumption gives
  \begin{equation}\label{Sigma33}
    e_{33} (\mathbf{u}) = - \frac{\nu}{1+\nu} \frac{1}{2G} (\sigma_{11} + \sigma_{22}) + \frac{\alpha (1-2\nu)}{2G (1+\nu)} p.
  \end{equation}
  \item For the other components of the strain tensor we have
  \begin{gather*}
  e_{11} ({\bf u} ) = \frac{\p u^\omega_1}{\p x_1}  - x_3 \frac{\partial^2 w}{\partial x_1^2}; \;
  e_{12} ({\bf u} ) = \frac{1}{2} (\frac{\p u^\omega_1}{\p x_2} + \frac{\p u^\omega_2}{\p x_1})  - x_3 \frac{\partial^2 w}{\partial x_1 \p x_2} ; \\
   e_{22} ({\bf u} ) = \frac{\p u^\omega_2}{\p x_2}  - x_3 \frac{\partial^2 w}{\partial x_2^2}; \;
   {  e_{33} ({\bf u} )= - \frac{\nu}{1-\nu} ( e_{11} ({\bf u} ) +  e_{22} ({\bf u} )) + \frac{\alpha (1-2\nu)}{2G (1-\nu)} p ;} \\
    {  \frac{2\nu G}{1-2\nu} \mbox{ div }{\bf u}- \alpha p  = \frac{2G\nu}{1-\nu} (e_{11} ({\bf u} ) + e_{22} ({\bf u} ))  -\frac{\alpha (1-2\nu)}{1-\nu} p }\label{trace}
\end{gather*}
\item For the other components of the stress tensor we have
\begin{gather}
    \sigma_{11} = \frac{2G}{1-\nu} (e_{11} ({\bf u} ) +\nu e_{22} ({\bf u} )) -\frac{\alpha (1-2\nu)}{1-\nu} p\label{sig11} \\
    \sigma_{22} = \frac{2G}{1-\nu} (e_{22} ({\bf u} ) +\nu e_{11} ({\bf u} )) -\frac{\alpha  (1-2\nu)}{1-\nu} p \label{sig22} \\
    \sigma_{12} = 2G e_{12} ({\bf u} ) = G (\frac{\p u^\omega_1}{\p x_2} + \frac{\p u^\omega_2}{\p x_1})  - 2G x_3 \frac{\partial^2 w}{\partial x_1 \p x_2} .\label{sig12}
\end{gather}\end{enumerate}
Next we define the {\bf stress resultants} (forces per unit length) $N_1 , N_2, N_{12}$ by
\begin{equation}\label{Resultants}
    N_1 =  \int^{\ell/2}_{-\ell/2} \sigma_{11} \ dx_3 , \ N_2 =  \int^{\ell/2}_{-\ell/2} \sigma_{22} \ dx_3 , \ N_{12} = \int^{\ell/2}_{-\ell/2} \sigma_{12} \ dx_3 ,
\end{equation}
{  the {\bf effective stress resultant due to the variation in pore pressure across the plate thickness} $N$ by}
\begin{equation}\label{Resp}
  {   N=- \int^{\ell/2}_{-\ell/2} p \ dx_3}
\end{equation}
and the {\bf external loading tangential to the plate}
\begin{equation}\label{load}
    f_i =  \sigma_{i3} |_{x_3 =\ell/2} - \sigma_{i3} |_{x_3 =-\ell/2} = \mathcal{P}^\ell_i +\mathcal{P}^{-\ell}_i , \; i=1,2.
\end{equation}
Next we average the equation (\ref{Coweq1a}) over the thickness and obtain
\begin{gather}
    \frac{\p N_1}{\p x_1} + \frac{\p N_{12}}{\p x_2} + f_1 =0 , \label{loads1} \\
    \frac{\p N_{12}}{\p x_1} +\frac{\p N_2}{\p x_2} + f_2 =0.\label{loads2}
\end{gather}
Inserting (\ref{sig11})-(\ref{sig12})
into the formulas (\ref{Resultants}) gives
\begin{gather}
   N_1  = \frac{2G\ell}{1-\nu} ( \frac{\p u^\omega_1}{\p x_1} + \nu \frac{\p u^\omega_2}{\p x_2} ) +\frac{\alpha (1-2\nu)  }{1-\nu} N ; \label{ResulCal1} \\
   N_2  = \frac{2G\ell }{1-\nu} ( \frac{\p u^\omega_2}{\p x_2} + \nu \frac{\p u^\omega_1}{\p x_1} ) +\frac{\alpha (1-2\nu) }{1-\nu} N ;\label{ResulCal2} \\
   N_{12}  = G\ell  (\frac{\p u^\omega_1}{\p x_2} + \frac{\p u^\omega_2}{\p x_1}).\label{ResulCal12}
\end{gather}
A substitution of  equations (\ref{ResulCal1})-(\ref{ResulCal12}) into the (\ref{loads1})-(\ref{loads2}) yields the equations for stretching of a plate of uniform thickness (\ref{Strech1}).

We need one more equation to complete the system (\ref{Strech1}). We have
\begin{gather}
{  \mbox{ div }{\bf u}  = \frac{1-2\nu}{1-\nu} \big(  \mbox{ div}_{x_1 , x_2} (u^\o_1 , u^\o_2 )  - x_3 \Delta_{x_1 , x_2} {  w}\big) + \frac{\alpha  (1-2\nu)}{2G (1-\nu)} p ,} \label{Stret0} \\
 \int^{\ell/2}_{-\ell/2}\mbox{ div }{\bf u} \ dx_3 = \frac{\ell (1-2\nu)}{1-\nu} \mbox{ div}_{x_1 , x_2} (u^\o_1 , u^\o_2 ) - \frac{\alpha  (1-2\nu)}{2G (1-\nu)} N . \label{Stretchaux}
\end{gather}
{
Next we average   equation (\ref{Coweq2}) over  thickness, use the assumption  (\ref{pressureapp}) and expression (\ref{Stretchaux}) and get (\ref{Strech2A}). Inserting (\ref{Stret0}) in equation (\ref{Coweq2}) and using hypothesis (\ref{pressureapp}) and equation (\ref{Strech2A}), yields equation (\ref{Stretch3}).}
\vskip5pt
It remains to find the equation for transverse deflection and for the bending moment due to the variation in pore pressure across the plate thickness.
\vskip1pt
  {  The  bending moment $M$ due to the variation in pore pressure across the plate thickness} is given by
\begin{equation}\label{pressurebendingmom}
    M=-  \int^{\ell/2}_{-\ell/2} x_3 p \ dx_3 .
\end{equation}
For the stress moments of the plate, which have as physical dimension the moment per unit length, we have:
\begin{enumerate} \item The {\bf twisting moment} $M_{12} :$
  \begin{equation}\label{M12}
     M_{12} = \int^{\ell/2}_{-\ell/2} \sigma_{12} \ x_3 dx_3 = - \frac{G\ell^3}{6} \frac{\p^2 w}{\p x_1 \p x_2}.
  \end{equation}
  \item The {\bf bending moments} $M_1$ and $M_2$:
  \begin{gather}
  M_1 = \int^{\ell/2}_{-\ell/2} \sigma_{11} \ x_3 dx_3 = - \frac{G\ell^3}{6(1-\nu )} (\frac{\p^2 w}{\p x^2_1 } +\nu \frac{\p^2 w}{\p x^2_2 } )  + \alpha \frac{1-2\nu }{1-\nu} M
  \label{M1} \\
  M_2 = \int^{\ell/2}_{-\ell/2} \sigma_{22} \ x_3 dx_3 = - \frac{G\ell^3}{6(1-\nu )} (\frac{\p^2 w}{\p x^2_2 } +\nu \frac{\p^2 w}{\p x^2_1 } )  + \alpha \frac{1-2\nu }{1-\nu} M.
  \label{M2}
  \end{gather}
  \item The {\bf transverse shear}  $(Q_1 , Q_2)$:
    $Q_i =\int^{\ell/2}_{-\ell/2} \sigma_{i3} \ d x_3 , \ i=1,2.$
\item {\bf Resultant external moment} $ (m_1 , m_2 ):$  $m_i =\frac{\ell}{2} (\mathcal{P}_i^\ell + \mathcal{P}_i^{-\ell} ).$ \end{enumerate}
Next we multiply the equation (\ref{Coweq1a}) by $x_3$ and integrate over the thickness to obtain
\begin{gather}
    \frac{\p M_1}{\p x_1} + \frac{\p M_{12}}{\p x_2} -Q_1 + m_1 =0 , \label{momes1} \\
    \frac{\p M_{12}}{\p x_1} +\frac{\p M_2}{\p x_2} -Q_2 + m_2 =0.\label{momes2}
\end{gather}
Averaging the the third equation in  (\ref{Coweq1a}) over the thickness yields:
\begin{equation}\label{momes3}
   \frac{\p Q_1}{\p x_1} + \frac{\p Q_{2}}{\p x_2}+ \mathcal{P}_3^\ell + \mathcal{P}_3^{-\ell}  =0
\end{equation}
Following Fung's textbook, we eliminate $Q_i$ from (\ref{momes1})-(\ref{momes3}) and obtain the equation of equilibrium in moments:\
\begin{equation}\label{eqmome}
    \frac{\p^2 M_1}{\p x_1^2} + 2\frac{\p^2 M_{12}}{\p x_1 \p x_2} + \frac{\p^2 M_2}{\p x_2^2} + \frac{\p m_1}{\p x_1} + \frac{\p m_{2}}{\p x_2}+ \mathcal{P}_3^\ell + \mathcal{P}_3^{-\ell}  =0.
\end{equation}
After inserting the formulas (\ref{M12})-(\ref{M2}) into (\ref{eqmome}), we obtain the poroelastic plate bending equation (\ref{bendingplate}).

\end{document}